\documentclass[10pt]{paper}%
\usepackage[T1]{fontenc}
\usepackage{amsmath,amssymb}
\usepackage{fullpage}
\usepackage{amsthm}
\usepackage{multirow}
\usepackage{color}
\usepackage{pifont}
\usepackage{amsmath}
\usepackage{graphicx}
  \usepackage[all]{xy}
\usepackage{amsfonts}
\usepackage{amssymb}
\usepackage{ytableau}%

\providecommand{\U}[1]{\protect\rule{.1in}{.1in}}
\providecommand{\U}[1]{\protect\rule{.1in}{.1in}}
\providecommand{\U}[1]{\protect\rule{.1in}{.1in}}
\providecommand{\U}[1]{\protect\rule{.1in}{.1in}}
\providecommand{\U}[1]{\protect\rule{.1in}{.1in}}
\newcommand{\ulambda}{{\boldsymbol{\lambda}}}

\newcommand{\umu}{{\boldsymbol{\mu}}}
\newcommand{\unu}{{\boldsymbol{\nu}}}
\newcommand{\Uglov}[2]{{\Phi}^{#1}_{#2}}
\newcommand{\uemptyset }{{\boldsymbol{\emptyset}}}
\newtheorem{Th}{Theorem}[subsection]
\newtheorem{lemma}[Th]{Lemma}
\newtheorem{Cor}[Th]{Corollary}
\newtheorem{Prop}[Th]{Proposition}

\theoremstyle{remark}
\newtheorem{Rem}[Th]{Remark}{\rmfamily}
\theoremstyle{definition}
\newtheorem{Def}[Th]{Definition}{\rmfamily}
\newtheorem{exa}[Th]{Example}{\rmfamily}
\newtheorem{abs}[Th]{\bfseries}

\newcommand\blfootnote[1]{%
  \begingroup
  \renewcommand\thefootnote{}\footnote{#1}%
  \addtocounter{footnote}{-1}%
  \endgroup
}
\begin{document}

\title{Kleshchev multipartitions and extended  Young diagrams }
\author{Nicolas Jacon}
\maketitle
\date{}
\blfootnote{\textup{2010} \textit{Mathematics Subject Classification}: \textup{20C08,05E10,17B37}} 
\begin{abstract}
We give a new simple  characterization of the set of   Kleshchev multipartitions, and more generally of the set of Uglov multipartitions. These combinatorial objects play an important role in various areas of representation theory of quantum groups, Hecke algebras or finite reductive groups. 
As a consequence, we obtain a proof of a generalization of a conjecture by Dipper, James and Murphy and a generalization of the LLT algorithm for arbitrary level.

\end{abstract}

\section{Introduction}

In the $1990s$, the works of Lascoux, Leclerc, Thibon \cite{LLT} and  Ariki \cite{A} have shown the existence of strong connections between the representation theory of Ariki-Koike algebras and the representation theory of quantum groups.  In particular, Ariki's categorification theorem has shown that  the decomposition matrices of Ariki-Koike algebras in characteristic $0$ can be computed using the Kashiwara-Lusztig  canonical bases in affine type $A$. More recently, 
 even more connections have been established by the works of Rouquier \cite{Rou}, Khovanov, Lauda \cite{KL}, Brundan and Kleshchev \cite{BK} on quiver Hecke algebras.

 One important feature in this theory is a combinatorial object known as Uglov multipartition.
  These objects  
  naturally index both the simple modules of Ariki-Koike algebras at a root of unity (by the work of S. Ariki and A. Mathas \cite{AM} and by the work of M. Geck and the author \cite{GJ}) and the crystal bases for irreducible highest weight modules in affine type $A$ (see \cite[Ch. 6]{GJ} for details). Unfortunately, in general, only a recursive definition of them   is known. It is based on the construction of a combinatorial graph called the crystal graph coming from the crystal basis theory for quantum groups. The main aim of this paper is to give a new simple characterization of Uglov multipartitions and, in particular, of a special case of them known as   Kleshchev multipartitions.  
  
  Let us be more precise. Let $(l,n)\in \mathbb{Z}_{>0}^2$. Let $W$ be the complex reflection group of type $G(l,1,n)$,  the wreath product  $\mathbb{Z}/l\mathbb{Z}\wr \mathfrak{S}_n$.  Let ${\bf s}=(s_1,\ldots,s_l)\in \mathbb{Z}^l$. Then one can attach to these data, the Hecke algebra of $W$ (or Ariki-Koike algebra). It is the $\mathbb{C}$-algebra with a presentation by generators $T_0, T_1,\ldots , T_{n-1}$ and relations $(T_0-\eta^{s_1})\ldots (T_0-\eta^{s_l})=(T_i-\eta)(T_i+1)=0$ for $i=1,\ldots,n-1$ and where $\eta\in \mathbb{C}^{*}$.  Let us denote by $e\in \mathbb{Z}_{>0}\sqcup \{\infty\}$ the order of $\eta$ as a root of $1$. 
If  $e>1$, this algebra is non semisimple and,   by Ariki's categorification theorem,
  the associated decomposition matrix is given by  the Kashiwara-Lusztig canonical basis for an irreducible highest weight $\mathcal{U} (\widehat{\mathfrak{sl}}_e)$-module. It is known that this canonical basis can be naturally indexed by the vertices of the crystal graph of the module. We obtain a labelling of the simple modules for the Ariki-Koike algebra by certain $l$-tuple of partitions (or multipartitions) known as Uglov $l$-partitions, depending on ${\bf s}$ and $e$ and constructed recursively on $n$ using this combinatorial graph (see the Definition in \S \ref{ra}.) 
When $s_{i}-s_{i-1}>n-1-e$ for all $i\in \{1,\ldots,l\}$, the set of Uglov multipartitions 
  is known as the set of Kleshchev multipartitions. Historically, it is these special cases that have first appeared in the context of the representation theory of Ariki-Koike algebras.   However, we now know that all classes of Uglov multipartitions  have a natural interpretation in the representation theory of Hecke algebras using the notions of basic sets     (see also \cite{Bo,DJM2} for the connection with the cellular structure of  Ariki-Koike algebras).

 These multipartitions and the associated crystal graphs not only appear in the context of the representation theory of Hecke algebras and quantum groups but also:
 \begin{itemize}
 \item in the context of the representation theory of rational Cherednik algebras:  one can identify the so called KZ-component of the ``Cherednik crystal of irreducible representations'' with these multipartitions (see for example \cite{GO,CGG})  
\item in the context of  the representation theory for finite reductive groups (see \cite{Gm}). For $l=2$, Uglov $2$-partitions are known to label some unipotent modules of the unitary group, and the associated graph has a remarkable  interpretation in  Harish-Chandra theory \cite{GJH}. 
\end{itemize}
All these connections give strong motivations for studying theses types of multipartitions.    
  If $l=1$, the $l$-partitions may be identified with the partitions and 
 the Uglov $l$-partitions have then a very simple definition: they are given by the $e$-regular partitions (that is the partition when no nonzero part are repeated $e$ or more times.) In contrast, when $l>1$, the definition is far more complicated. 
  If $l=2$, S. Ariki, V. Kreiman, and S. Tsuchioka \cite{AKT} have    given an alternative non recursive definition of the set of Kleshchev $2$-partitions using abaci display. In the general case, a simple  characterization  was still missing.

   In this paper, we give a new simple characterization of the set of Uglov  $l$-partitions for all $l\in \mathbb{Z}_{>0}$ and all ${\bf s}\in \mathbb{Z}^l$: see Theorem \ref{main}. It thus also concerns the set of Kleshchev multipartitions, as a special case (see \S \ref{kl}.) This characterization  is still recursive on $n$ but  easier than the original definition and it does not use the crystal graph.  The proof will then be largely combinatorial. It is based on extensions of classical combinatorial definitions around the combinatorics of Young diagrams  and on the study of certain crystal isomorphisms already introduced in  \cite{JL}. In the context of Cherednik algebras, such isomorphisms can be interpreted as wall crossing functors in the sense of Losev \cite{Lo}  as it is shown in \cite{JLwc}.

 We also develop two consequences of our main result. The first one is a proof of a generalization of a conjecture by Dipper, James and Murphy \cite{DJM} stated by Graham and Lehrer in \cite[\S 5]{GL} concerning the set of Kleshchev multipartitions. A proof has been previously given in the case $l=2$ by S. Ariki and the author \cite{AJ}, for $l\in \mathbb{N}$ and  $e=\infty$ by J. Hu \cite{Hu}, and for $e=2$ by J. Hu, K. Zhou and K. Wang \cite{HZW}. We here treat the most general case $l\in \mathbb{N}_{>0}$ and $e\in \mathbb{N}_{>1}$ (without using these previous works.) The second consequence of our main result is  a direct generalization of the LLT algorithm computing the canonical bases of irreducible highest weight  $\mathcal{U} (\widehat{\mathfrak{sl}}_e)$-modules (see \S \ref{llt}). Previously, it was only available for  $l=1$ \cite{LLT} (see however \cite{LLTa} which proposes a general analogue) and  we again treat the most general case  $l\in \mathbb{N}_{>0}$.

The paper will be organized as follows. The second, the third and the forth sections present the main definitions we need to state the main result which is exposed in the fifth section. This t section also explains the strategy for the proof of the main theorem. We then  section explores a particular case of Uglov $l$-partitions which is  needed in the sequel: The FLOTW $l$-partitions. We explain and study previous results on crystal isomorphisms which are crucial for the proof of our Theorem. The proof is finalized in  the ninth  section. Then the last section gives  two consequences of our result: first  the proof of the generalized Dipper-James-Murphy conjecture. Then we  show how the generalization of the LLT algorithm is deduced from our results.    \\
\\
  {\bf Acknowledgements.} The  author is supported by Agence National de la Recherche  Acort ANR-12-JS01-0003 and  GeRepMod ANR-16-CE40-0010-01. The author thanks Chris Bowman and Thomas Gerber for useful discussions. He also thanks the referee for a lot of helpful comments and suggestions.

\section{Extended Young diagrams}
In this section, we give several necessary definitions concerning the combinatorics of 
 partitions and multipartitions. Most of them are generalizations of well-known 
  definitions. We illustrate   the new notions with examples.

 \subsection{Partitions and multipartitions}
 
\begin{abs}
Recall that a {\it partition} $\lambda$ of rank $n\in \mathbb{Z}_{>0}$ is a sequence  of non negative and non increasing integers 
 $\lambda=(\lambda_1,\ldots,\lambda_r)$ such that $\sum_{1\leq i\leq r} \lambda_i =n$. The integer $n$ is called the rank of the partition.  The integers $\lambda_i$ (for $1\leq i\leq r$) are called the parts of the partition. 
 By a slight abuse of  notation, we admit that one can add or delete to a partition 
as many part $0$ as we want without changing it.
 The empty partition
  is denoted by $\emptyset$ (and is thus identified with the partitions $(0)$, $(0,0)$, etc.) 
We denote by $\ell (\lambda)$ the minimal integer  
such that $\lambda_{\ell (\lambda)}=0$ with the convention that $\ell(\emptyset)=0$. Let $l\in \mathbb{Z}_{>0}$,
  an {\it $l$-partition} $\ulambda$ of $n$ is an $l$-tuple $(\lambda^1,\ldots,\lambda^l)$ 
   of partitions  such that the sum of the rank of the partitions 
    $\lambda^j$ with $1\leq j\leq l$ is $n$.  The integer $n$ is called the rank of the $l$-partition.  The unique $l$-partition $(\emptyset,\ldots,\emptyset)$ of rank $0$ is denoted by $\uemptyset$. We denote by $\Pi^l (n)$ 
 the set of all $l$-partitions of rank $n$ and by $\Pi^l$ the set of all $l$-partitions. 

\end{abs}
\begin{abs}\label{ext}
Let $l\in \mathbb{Z}_{>0}$. Let $\ulambda$ be an $l$-partition. The {\it nodes} of the $l$-partition $\ulambda$ are the elements 
$(a,b,c)$ where $a\in \{1,\ldots,\ell(\lambda^c)\}$, $b\in \{1,\ldots,\lambda_a^c\}$  and 
$c\in  \{1,\ldots, l\}$. The set of all nodes of $\ulambda$ is called the {\it  Young diagram} of $\ulambda$  and it is denoted by 
$\mathcal{Y} (\ulambda)$.

 The {\it extended nodes} of the $l$-partition $\ulambda$ are the 
 following elements of $\mathbb{Z}_{\geq 0} \times \mathbb{Z}_{\geq 0}\times \{1,\ldots, l\}$:
\begin{enumerate}
\item the elements of $\mathcal{Y} (\ulambda)$,
\item the elements of the form $(0,b,c)$ where $b>\lambda^c_1$ and $c\in  \{1,\ldots, l\}$,
\item the elements of the form  $(a,0,c)$ where $a>\ell (\lambda^c)$ and $c\in \{1,\ldots,l\}$.
\end{enumerate}
The set of extended nodes of $\ulambda$  is called the {\it extended Young diagram} of $\ulambda$  and it is denoted by $\mathcal{Y}^{\textrm{ext}} (\ulambda)$. The nodes which are in the extended Young diagram but not in the Young diagram are called {\it virtual nodes}. So the extended nodes consist in the set of nodes together with the set of virtual nodes. 

 The extended Young diagram has infinite cardinal and  contains
 the usual Young diagram  $\mathcal{Y} (\ulambda)$.
   It is convenient to represent it as an $l$-tuple of array of (infinite) boxes as 
 in the case of Young diagram.  
  \end{abs}
  \begin{exa}
  Let $l=3$ and take the $3$-partition $(3.1.1,1,2.1)$ of $n=9$. The  extended Young diagram is given as follows. 
  
\vspace{1cm}  
  
  \centerline{
 \Bigg(\  \ \ \ \ytableausetup
{mathmode}
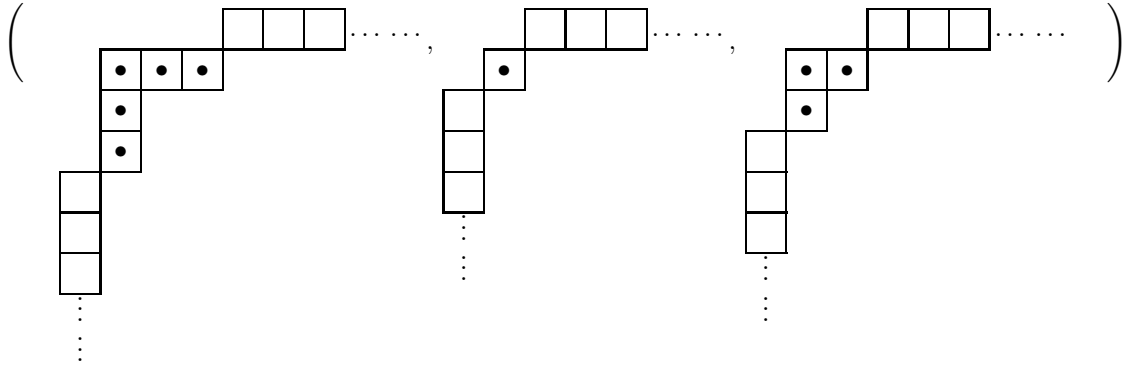
\begin{ytableau}
\none  &  \none  & \none  & \none  & \ & \ &\ & \none[\dots]& \none[\dots] \\
\none  &  \bullet & \bullet & \bullet \\
\none & \bullet \\
\none & \bullet \\  
\  \\
\  \\
\  \\
\none[\vdots] \\
 \none[\vdots] \\
\end{ytableau},
\begin{ytableau}
\none  &  \none  & \ & \ &\ & \none[\dots]& \none[\dots] \\
\none  &  \bullet \\  
\  \\
\  \\
\  \\
\none[\vdots] \\
 \none[\vdots] \\
\end{ytableau},
\begin{ytableau}
\none  &  \none  & \none   & \ & \ &\ & \none[\dots]& \none[\dots] \\
\none  &  \bullet & \bullet  \\
\none & \bullet \\  
\  \\
\  \\
\  \\
\none[\vdots] \\
 \none[\vdots] \\
\end{ytableau}\  \ \ \  \Bigg)}

 The boxes containing a bullet above correspond to the boxes of the usual Young diagram.
 \end{exa}
 \subsection{Order on nodes}
\begin{abs} The combinatoric that we now explain is adapted to the study of the representation theory of Fock spaces. We will explain later in which way. 
Let $e\in \mathbb{Z}_{>1}$  and let ${\bf s}=(s_1,\ldots,s_l)\in \mathbb{Z}^l$. One can attach to each extended node $\gamma=(a,b,c)$ of the extended Young diagram its {\it content} (depending  on the choice of ${\bf s}$):  
$$\operatorname{cont} (\gamma)=b-a+s_c\in \mathbb{Z}.$$
We denote also $\operatorname{comp} (\gamma)=c$. 
By definition, the {\it residue} (depending  on the choice of ${\bf s}$ and $e$) $\operatorname{res} (\gamma)$ of the extended node $\gamma$ is the content modulo $e$. Throughout the paper, $\mathbb{Z}/e\mathbb{Z}$ will be identified with $\{0,\ldots,e-1\}$. 

If $\operatorname{res} (\gamma)=\mathfrak{j}$ then we say that $\gamma$ is a (extended) $\mathfrak{j}$-node. Again, it is convenient to represent the content  of each extended node of an $l$-partition in the associated box of  the (representation of the) associated extended Young diagram as in the next example.

The {\it boundary} of the extended Young diagram is by definition the subset  of $\mathcal{Y}^{\textrm{ext}} (\ulambda)$ consisting of the elements $(a,b,c)\in \mathcal{Y} (\ulambda)$ such that $(a,b+1,c)$ or $(a+1,b,c)$ are not in 
 $\mathcal{Y}^{\textrm{ext}} (\ulambda)$. The {\it vertical boundary} of the extended Young diagram is by definition the subset  of $\mathcal{Y}^{\textrm{ext}} (\ulambda)$ consisting of the elements $(a,b,c)\in \mathcal{Y} (\ulambda)$ such that $(a,b+1,c)$ is not in $\mathcal{Y}^{\textrm{ext}} (\ulambda)$. So it contains  nodes of the Young diagram: the nodes of the vertical boundary,  together with the virtual nodes  of type $(3)$ in \S \ref{ext}: the virtual nodes of the vertical boundary.  
 
 The {\it horizontal boundary} of the extended Young diagram is by definition the subset  of $\mathcal{Y}^{\textrm{ext}} (\ulambda)$ consisting of the elements $(a,b,c)\in \mathcal{Y} (\ulambda)$ such that $(a+1,b,c)$ is not in $\mathcal{Y}^{\textrm{ext}} (\ulambda)$.  So it contains some nodes of the Young diagram: the nodes of the horizontal boundary,  together with the virtual nodes  of type $(2)$ in \S \ref{ext}: the virtual nodes of the horizontal boundary.

\end{abs}
\begin{exa}
  Let $l=3$ and take the $3$-partition $\ulambda=(3.1.1,1,2.1)$ with ${\bf s}=(0,1,4)$. The  extended Young diagram 
   with content is:
  
\vspace{1cm}  
  
  \centerline{\small{
 \Bigg(\  \ \ \ \ytableausetup
{mathmode}
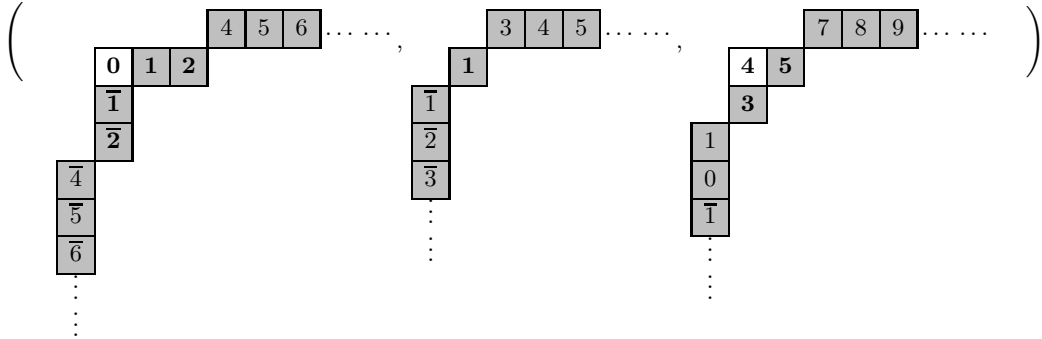
\begin{ytableau}
\none  &  \none  & \none  & \none  & *(lightgray){4} & *(lightgray){5} & *(lightgray){6} & \none[\dots]& \none[\dots] \\
\none  &  {\bf 0} & *(lightgray){\bf 1} & *(lightgray){\bf 2} \\
\none & *(lightgray)\overline{\bf 1} \\
\none & *(lightgray)\overline{\bf 2} \\  
*(lightgray)\overline{4}  \\
*(lightgray)\overline{5}  \\
*(lightgray)\overline{6}  \\
\none[\vdots] \\
 \none[\vdots] \\
\end{ytableau},
\begin{ytableau}
\none  &  \none  & *(lightgray)3 & *(lightgray)4 & *(lightgray)5 & \none[\dots]& \none[\dots] \\
\none  &  *(lightgray){\bf 1} \\  
*(lightgray)\overline{1} \\
*(lightgray)\overline{2}  \\
*(lightgray)\overline{3}  \\
\none[\vdots] \\
 \none[\vdots] \\
\end{ytableau},
\begin{ytableau}
\none  &  \none  & \none   & *(lightgray)7 & *(lightgray)8 & *(lightgray)9 & \none[\dots]& \none[\dots] \\
\none  &  {\bf 4} &*(lightgray){\bf 5}  \\
\none & *(lightgray){\bf 3} \\  
*(lightgray)1  \\
*(lightgray)0  \\
*(lightgray)\overline{1}  \\
\none[\vdots] \\
 \none[\vdots] \\
\end{ytableau}\  \ \ \  \Bigg)}}

The contents in bold correspond to the content of the nodes of the usual Young diagram. The boxes are colored in  gray if they correspond to extended nodes. The notation $\overline{j}$ stands  for  $-j$. 
\begin{itemize}
\item The vertical boundary contains:
\begin{itemize}
\item The nodes: $(2,1,1)$, $(3,1,1)$, $(1,3,1)$, $(1,1,2)$, $(1,2,3)$, $(2,1,3)$. 
\item The virtual nodes: $(a_1,0,1)$ for $a_1\geq 4$, $(a_2,0,2)$ for $a_2\geq 2$, $(a_3,0,3)$ for $a_3\geq 3$.  
\end{itemize}
\item The horizontal boundary contains:
\begin{itemize}
\item The nodes $(1,2,1)$, $(1,3,1)$, $(3,1,1)$, $(1,1,2)$, $(1,2,3)$, $(2,1,3)$. 
\item The virtual nodes $(0,b_1,1)$ for $b_1\geq 4$, $(0,b_2,2)$ for $b_2\geq 2$, $(0,b_3,3)$ for $b_3\geq 3$.  
\end{itemize}
\end{itemize}

\end{exa}

\begin{abs}\label{notimp}
A node $\gamma=(a,b,c)$ of $\mathcal{Y} (\ulambda)$ is said to be {\it removable} for $\ulambda$  if  
 $\mathcal{Y} (\ulambda)\setminus \{\gamma\}$ is the Young tableau 
  of a $l$-partition $\umu$.  If $\gamma=(a,b,c)\in \mathbb{Z}_{> 0} \times \mathbb{Z}_{> 0}\times \{1,\ldots, l\}$ is such that $\mathcal{Y} (\ulambda)\sqcup \{\gamma\}$ is the Young tableau of 
   a $l$-partition $\umu$ then it is said to be {\it addable} for $\ulambda$.  
  
Note that the intersection between the vertical and the horizontal boundary is given by the removable nodes.

It is also important to  remark that, given a content and a component of a multipartition $\ulambda$, there always exist one unique extended node in this component with the given content which is either addable, either in the boundary of $\ulambda$ . We will denote by $\mathcal{E}_{\mathfrak{j}} (\ulambda)$ the set consisting of :
\begin{itemize}
\item addable $\mathfrak{j}$-nodes of $\ulambda$,
\item  extended $\mathfrak{j}$-nodes of the boundary of  $\ulambda$. 
\end{itemize}
\end{abs}
\begin{exa}
We keep the previous example and take $e=3$. Then the set $\mathcal{E}_{0} (\ulambda)$ consists in extended  nodes of the three components of $\ulambda$. For example, the extended nodes in  $\mathcal{E}_{0} (\ulambda)$  in the component $\lambda^1$ are:
\begin{itemize}
\item  the extended nodes of the horizontal boundary $(0,3k,1)$ for $k\geq 2$ of content $3k$,
\item the addable node $(1,4,1)$ of content $3$, the addable node $(2,2,1)$ of content $0$ 
 and the addable node $(4,1,1)$ of content $-3$, 
\item the extended noded of the vertical boundary $(3k,0,1)$ with $k\geq 2$ of content $-3k$. 
\end{itemize}
One can see that for each $j\in \mathbb{Z}$ in the class of $0$ modulo $3$, there exists a unique extended node of $\lambda^1$   in $\mathcal{E}_{0} (\ulambda)$  with content $j$. 
\end{exa}

\begin{abs}\label{refor} 
We will now define a total order on the elements of $\mathcal{E}_{\mathfrak{j}} (\ulambda)$   with the same residue.  This order will be crucial for the definition of staggered multipartitions in the next section. It will be the key to understand the link of our approach to 
 the old one.
 Let $\gamma_1=(a_1,b_1,c_1) \in \mathcal{E}_{\mathfrak{j}} (\ulambda)$ and 
  $\gamma_2=(a_2,b_2,c_2) \in \mathcal{E}_{\mathfrak{j}} (\ulambda)$. Then we write $\gamma_1<\gamma_2$ if 
  \begin{itemize}
  \item $\operatorname{cont}(\gamma_1)<\operatorname{cont}(\gamma_2)$ or,
\item $\operatorname{cont}(\gamma_1)=\operatorname{cont}(\gamma_2)$ and $c_1>c_2$.

  \end{itemize}
Note that this order strongly depends on the choice of ${\bf s}$. 
We see that we have  $\gamma_1<\gamma_2$ and that these two extended nodes are consecutive if we are in one of the following two cases:
\begin{itemize}
\item $\operatorname{cont}(\gamma_1)=\operatorname{cont}(\gamma_2)$ and $c_1=c_2+1$,
\item $\operatorname{cont}(\gamma_1)+e=\operatorname{cont}(\gamma_2)$,  $c_1=1$, $c_2=l$. 
\end{itemize}
We denote in this case $\gamma_1<_{co} \gamma_2$.
\end{abs}

\section{Staggered multipartitions}

In this section, using the order that we have defined above, we give the definition of a certain  class of multipartitions: the staggered multipartitions. This definition depends on the choice of  ${\bf s}\in \mathbb{Z}^l$. We then simplfiy this definition in the case where the $l$-tuple ${\bf s}$ satisfy a certain condition called ``asymptotic''. 

\subsection{Definition of staggered multipartitions}

\begin{Def} Let $\ulambda$ be an $l$-partition. 
Let $\mathfrak{j}\in \mathbb{Z}/e\mathbb{Z}$. 
A sequence $(\gamma_i)_{i\in \mathbb{Z}_{>0}}$ of extended $\mathfrak{j}$-nodes of $\ulambda$  is said to be a {\it staggered sequence} if this is a sequence of extended $\mathfrak{j}$-nodes of the boundary of $\ulambda$ such that:
\begin{itemize}
\item We have for all $i\in \mathbb{Z}_{>0}$, $\gamma_{i} <_{co} \gamma_{i+1}$. 
\item There exists $s\in \mathbb{Z}_{>0}$ such that $\gamma_s$ is a removable node.
\item  All the $\mathfrak{j}$-nodes $\gamma$ in the boundary such that 
 $\gamma< \gamma_1$  are virtual nodes of the vertical boundary  or  there is an addable $\mathfrak{j}$-node $\gamma$ such that $\gamma<_{co} \gamma_1$. 
 \end{itemize}
\end{Def}
Note that, keeping the above notations, a staggered sequence of (extended) $\mathfrak{j}$-nodes is thus an infinite  sequence of extended nodes in $\mathcal{E}_{\mathfrak{j}} (\ulambda)$. 

\begin{Def}
An $l$-partition $\ulambda$ is said to be a {\it staggered $l$-partition} if $\ulambda=\uemptyset$ or if 
\begin{itemize}
\item There exists $\mathfrak{j}\in \mathbb{Z}/e\mathbb{Z}$ such that 
 $\ulambda$ admits a staggered  sequence $(\gamma_i)_{i\in \mathbb{Z}_{>0}}$ of extended $\mathfrak{j}$-nodes.
 \item If we delete all the removable $\mathfrak{j}$-nodes of this sequence, the resulting $l$-partition is  a staggered $l$-partition. 

\end{itemize}
\end{Def}
\begin{Def}
Let $\ulambda$  be a staggered $l$-partition with rank $n\in \mathbb{Z}_{>0}$. Then $\ulambda$ admits a 
  staggered sequence of nodes. Let  
 $(\gamma_1,\ldots,\gamma_r)$ be the removable nodes of this sequence. We denote by $\mathfrak{j}\in \mathbb{Z}/e\mathbb{Z}$ the common residue of these nodes. Let $\ulambda'$ be the $l$-partition of $n-r$ obtained by removing the nodes 
 $(\gamma_1,\ldots,\gamma_r)$ from $\ulambda$. Then a {\it staggered sequence of residues} of $\ulambda$ is the concatenation of the sequence $(\underbrace{\mathfrak{j},\ldots,\mathfrak{j}}_{r})$ with a staggered sequence of residues of
$\ulambda'$. 
\end{Def}

\begin{exa}
Let $l=2$, $e=3$, ${\bf s}=(0,4)$ and $\ulambda=(4.1,3.2.1.1)$
\vspace{1cm}  
  
  \centerline{\small{
 \Bigg(\  \ \ \ \ytableausetup
{mathmode}
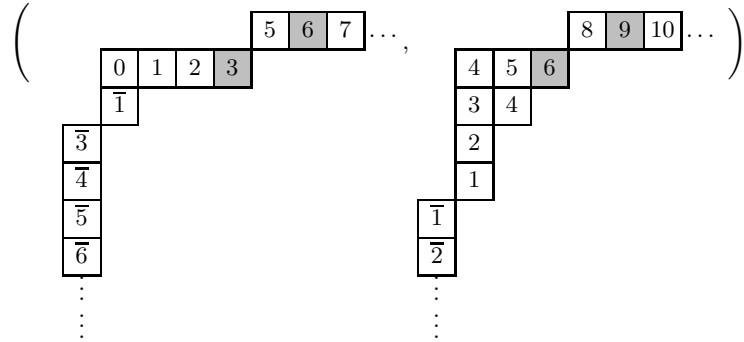
\begin{ytableau}
\none  &  \none  & \none  & \none  &  \none& {5} & *(lightgray){6} & 7& \none[\dots] \\
\none  &  { 0} & { 1} & { 2} & *(lightgray)3 \\
\none & \overline{ 1} \\
\overline{3}  \\  
\overline{4}  \\
\overline{5}  \\
\overline{6}  \\
\none[\vdots] \\
 \none[\vdots] \\
\end{ytableau},
\begin{ytableau}
\none  &  \none  & \none  & \none  &  {8} & *(lightgray){9} & 10& \none[\dots] \\
\none  &  { 4} & { 5} & *(lightgray){ 6} \\
\none & 3 & 4 \\
\none & 2   \\  
\none & 1  \\
\overline{1}  \\
\overline{2}  \\
\none[\vdots] \\
 \none[\vdots] \\
\end{ytableau}
 \Bigg)}}
One can start with the staggered sequence coloured in gray in the above diagram (where each extended node in the diagram comes with its content.) Removing the removable $3$-nodes  of this sequence, we obtain the bipartition $(3.1,2.2.1.1)$. 

\vspace{1cm}

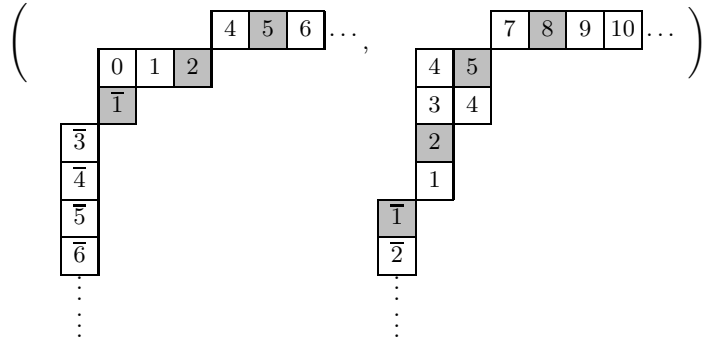
  \centerline{\small{
 \Bigg(\  \ \ \ \ytableausetup
{mathmode}\begin{ytableau}
\none  &  \none  & \none  & \none   & {4} & *(lightgray){5} &  6& \none[\dots] \\
\none  &  { 0} & { 1} & *(lightgray){ 2}  \\
\none & *(lightgray)\overline{ 1} \\
\overline{3}  \\  
\overline{4}  \\
\overline{5}  \\
\overline{6}  \\
\none[\vdots] \\
 \none[\vdots] \\
\end{ytableau},
\begin{ytableau}
\none  &  \none  & \none  & 7  &  *(lightgray){8} & {9} & 10& \none[\dots] \\
\none  &  { 4} &*(lightgray) { 5}   \\
\none & 3 & 4 \\
\none & *(lightgray)2   \\  
\none & 1  \\
*(lightgray)\overline{1}  \\
\overline{2}  \\
\none[\vdots] \\
 \none[\vdots] \\
\end{ytableau}
 \Bigg)}}
We obtain the bipartition $(2,2.2.1.1)$ and then:
\vspace{1cm}

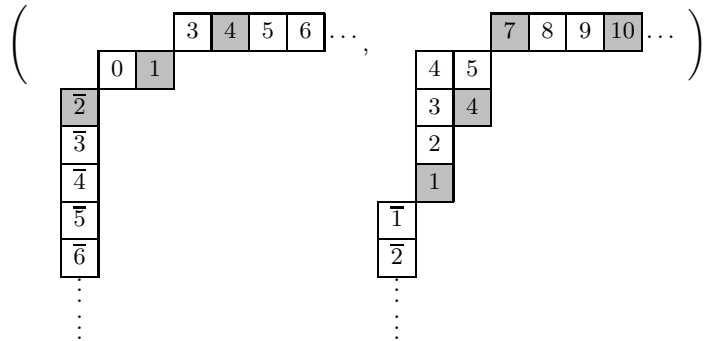
  \centerline{\small{
 \Bigg(\  \ \ \ \ytableausetup
{mathmode}\begin{ytableau}
\none  &  \none  & \none  & 3   & *(lightgray){4} & {5} &  6& \none[\dots] \\
\none  &  { 0} & *(lightgray){ 1}    \\
*(lightgray)\overline{2}  \\
\overline{3}  \\  
\overline{4}  \\
\overline{5}  \\
\overline{6}  \\
\none[\vdots] \\
 \none[\vdots] \\
\end{ytableau},
\begin{ytableau}
\none  &  \none  & \none  & *(lightgray)7  &  {8} & {9} & *(lightgray){10}& \none[\dots] \\
\none  &  { 4} & { 5}   \\
\none & 3 &  *(lightgray) 4 \\
\none & 2   \\  
\none & *(lightgray) 1  \\
\overline{1}  \\
\overline{2}  \\
\none[\vdots] \\
 \none[\vdots] \\
\end{ytableau}
 \Bigg)}}

We obtain the bipartition $(1,2.1.1)$. Continuing in this way, we get $(1,1.1)$, then $(\emptyset,1)$ and then $(\emptyset,\emptyset)$ which shows that we have a staggered bipartition. The associated staggered sequence of residues is:
$$(3,3,2,2,1,1,1,2,2,0,0,1).$$
\end{exa}
\begin{Rem}
Assume that $l=1$, $e\in \mathbb{Z}_{>1}$. One can easily show that the set of  staggered partitions correspond the set of $e$-regular partitions, that is the set of partitions $\lambda$ where no non zero parts are repeated $e$ or more times. This will be in fact a consequence of our main theorem. 

\end{Rem}

\subsection{Staggered multipartitions in the asymptotic case}

\begin{abs}\label{kl}
 Fix $n\in \mathbb{Z}_{>0}$. We will assume that the following hypothesis is satisfied for ${\bf s}=(s_1,\ldots,s_l)\in \mathbb{Z}^l$ (known as the asymptotic case):
\begin{equation*}\label{hs} 
 {\bf (H)}\qquad \forall i\in \{2,\ldots,l\},\ s_{i}-s_{i-1}\geq n-1+e.
 \end{equation*} 
 Let $\ulambda\in \Pi^l$. 
 By \cite[Ex. 6.2.16]{GJ}, the order $<$ on $\mathcal{E}_{\mathfrak{j}} (\ulambda)$ (for $\mathfrak{j}\in \mathbb{Z}/e\mathbb{Z}$)  that we have 
  defined above has the following alternative  description $\gamma=(a,b,c)$ and $\gamma'=(a',b',c')$ in $\mathcal{E}_{\mathfrak{j}} (\ulambda)$ satisfy 
   $\gamma<\gamma'$ if and only if $c'<c$ or if $c=c'$ and $a'<a$. This is the order used in particular in \cite{DJMa}.

Using this remark, one can simplify the definition of staggered $l$-partitions in this case. Indeed, if $\gamma$ is a removable node of the Young diagram of $\ulambda$ and if $\gamma>_{co} \gamma'$, then $\gamma'$ must be a virtual node of the horizontal boundary, and thus not removable. 
 Let $\mathcal{K}_\mathfrak{j}(\ulambda)$ denote the set of the nodes of the boundary of $\ulambda$ and the addable nodes of $\ulambda$ of the same residue $\mathfrak{j}$. 
 Define  a binary relation on $\mathcal{K}_\mathfrak{j}(\ulambda)$  as follows (this thus concerns only the nodes of the Young diagram, not the virtual ones). For $\gamma=(a,b,c)$ and $\gamma'=(a',b',c')$ 
 in $\mathcal{K}_\mathfrak{j} (\ulambda)$, we denote $\gamma=(a,b,c)<_\mathcal{K} \gamma'=(a',b',c')$ if
\begin{itemize}
\item $c=c'$ and $\textrm{cont} (\gamma)=\textrm{cont}(\gamma')-e$ or
\item $c'=c+1$ and $\gamma'$ (resp. $\gamma$) has minimal (resp. maximal) content among  the nodes of 
 $\mathcal{K}_{\mathfrak{j}} (\ulambda)$ in the component $\lambda^{c'}$ (resp. $\lambda^{c}$).
\end{itemize}
It follows from the definition that for two $\mathfrak{j}$-nodes $\gamma$ and $\eta$ in  $\mathcal{K}_\mathfrak{j}(\ulambda)$,  if we have $\gamma<_{\mathcal{K}} \eta$ then there exists a sequence $(\eta_i)_{i=1,\ldots,r}$ of 
 virtual nodes of the boundary of $\ulambda$ such that   
$\gamma<_{co}\eta_1<_{co}\ldots <_{co}\eta_r<_{co} \eta$. As virtual nodes are not removable, we can thus give another definition of staggered $l$-partitions in this case, which does not use the notion of extended nodes.

A sequence $(\gamma_i)_{i=1,\ldots,r}$ of $\mathfrak{j}$-nodes of $\mathcal{K}_{\mathfrak{j}}(\ulambda)$ is said to be a $\mathcal{K}$-staggered sequence if:
\begin{itemize}
\item This is a sequence of nodes of the boundary of $\ulambda$ containing at least one removable node.
\item We have for all $i\in \{1,\ldots,r-1\}$, $\gamma_{i} <_\mathcal{K} \gamma_{i+1}$, 
\item There are no elements $\eta$ in $\mathcal{K}_{\mathfrak{j}}(\ulambda)$  such that $\eta>_\mathcal{K} \gamma_r$.
\item If there exists $\eta$ in $\mathcal{K}_{\mathfrak{j}}(\ulambda)$  such that $\eta<_\mathcal{K} \gamma_1$ 
 then $\eta$ is addable.
\end{itemize}
Under the hypothesis ${\bf (H)}$, we obtain that a $l$-partition $\ulambda$ is a {\it staggered $l$-partition} (with respect to $e$ and ${\bf s}$ as above)  if $\ulambda=\uemptyset$ or if 
\begin{itemize}
\item There exists $\mathfrak{j}\in \mathbb{Z}/e\mathbb{Z}$ such that 
 $\ulambda$ admits a  $\mathcal{K}$-staggered  sequence  of $\mathfrak{j}$-nodes.
 \item If we delete all the removable $\mathfrak{j}$-nodes of this sequence, the resulting $l$-partition is staggered. 
\end{itemize}
In addition, assume that $\ulambda$ is a staggered $l$-partition. Then it is clear that, taking successively the residues of the removable nodes of the $\mathcal{K}$-staggered 
 sequences of the staggered $l$-partitions we get in this recursive process, we obtain the staggered sequence of residues of $\ulambda$. 
\end{abs}
\begin{exa}
Assume that $l=2$, $(s_1,s_2)=(14,0)$, $e=3$. For $n\leq 9$, 
 we satisfy the above hypotheses. Let us consider the bipartition $(2.2,3.1.1)$. 
 One can check that this is a staggered bipartition. Below is the sequence of the Young diagram 
  of the associated staggered bipartition in the recursive process. We color in grey the boxes of the nodes of the associated $\mathcal{K}$-staggered sequences.  
\vspace{1cm}

  \centerline{\small{
 \Bigg(\  \ \ytableausetup
{mathmode}
\begin{ytableau}
   {14} & {15}  \\
  { 13} & *(lightgray){ 14}    \\
\end{ytableau},
\begin{ytableau}
  {0} & {1}  & *(lightgray){2}\\
    *(lightgray){\overline{1}}  \\
\overline{2}  \\
\end{ytableau}
 \Bigg)} 
 $\to$
 \small{
 \Bigg(\  \ \ytableausetup
{mathmode}
\begin{ytableau}
   {14} & *(lightgray){15}  \\
  { 13}     \\
\end{ytableau},
\begin{ytableau}
  {0} & {1}  \\
    {\overline{1}}  \\
\overline{2}  \\
\end{ytableau}
 \Bigg)} 
  $\to$
  \small{
 \Bigg(\  \ \ytableausetup
{mathmode}
\begin{ytableau}
   {14}   \\
  *(lightgray){ 13}      \\
\end{ytableau},
\begin{ytableau}
  {0} & *(lightgray){1}  \\
   {\overline{1}}  \\
*(lightgray){\overline{2}}  \\
\end{ytableau}
 \Bigg)} 
  $\to$
  \small{
 \Bigg(\  \ \ytableausetup
{mathmode}
\begin{ytableau}
   *(lightgray){14}   
    \\
\end{ytableau},
\begin{ytableau}
  {0} \\
    *(lightgray){\overline{1}}  \\
\end{ytableau}
 \Bigg)}
   $\to$
  \small{
 \Bigg(\  \ytableausetup
{mathmode}
$\emptyset$\ ,
\begin{ytableau}
  *(lightgray){0} 
    \\
\end{ytableau}
 \Bigg)}  
  }
 
\vspace{0,3cm} 
 
\noindent  The staggered sequence of residues is $2,2,0,1,1,1,2,2,0$.

\end{exa}

\begin{exa}
Assume that $l=2$, $(s_1,s_2)=(12,0)$, $e=2$. For $n\leq 9$, 
 we satisfy the above hypotheses. Let us consider the bipartition $(2.2.1.1,\emptyset)$. 
 One can check that this is a staggered bipartition. Below is the sequence of the Young diagram 
  of the associated staggered bipartition in the recursive process. We color in grey the boxes of the nodes of the associated $\mathcal{K}$-staggered sequences.  
\vspace{1cm}

  \centerline{\small{
 \Bigg(\  \ \ytableausetup
{mathmode}
\begin{ytableau}
   {12} & {13}  \\
  { 11} & *(lightgray){ 12}    \\
  10 \\
  *(lightgray){ 9}
\end{ytableau},
$\emptyset $
 \Bigg)} 
 $\to$
 \small{
 \Bigg(\  \ \ytableausetup
{mathmode}
\begin{ytableau}
   {12} & *(lightgray){13}  \\
  { 11}     \\
   *(lightgray){10}
\end{ytableau},
$\emptyset$
 \Bigg)} 
  $\to$
 \small{
 \Bigg(\  \ \ytableausetup
{mathmode}
\begin{ytableau}
   {12} \\
  *(lightgray){11}   
\end{ytableau}\ ,
$\emptyset$
 \Bigg)}  
  $\to$
 \small{
 \Bigg(\  \ \ytableausetup
{mathmode}
\begin{ytableau}
   *(lightgray){12}   
\end{ytableau}\ ,
$\emptyset$
 \Bigg)} 
  }
 
\vspace{1cm} 
 
\noindent Take now ${\bf s}=(0,0)$ and the same bipartition, we search for successive staggered sequences :

\vspace{1cm}

   \centerline{\small{
 \Bigg(\  \ \ytableausetup
{mathmode}
\begin{ytableau}
   {0} & {1}  \\
  { \overline{1}} & *(lightgray){0}    \\
  \overline{2}\\
    \overline{3}
\end{ytableau},
$\emptyset $
 \Bigg)} 
 $\to$
 \small{
 \Bigg(\  \ \ytableausetup
{mathmode}
\begin{ytableau}
   {0} & *(lightgray){1}  \\
  \overline{ 1}     \\
   \overline{2}\\
    \overline{3}
\end{ytableau},
$\emptyset$
 \Bigg)} 
  $\to$
 \small{
  \Bigg(\  \ \ytableausetup
{mathmode}
\begin{ytableau}
   {0}    \\
  \overline{ 1}     \\
   \overline{2}\\
    \overline{3}
\end{ytableau},
$\emptyset$
 \Bigg)}
}

\vspace{1cm}

This last $2$-partition does not have any staggered sequences. Thus  $(2.2.1,\emptyset)$ is not a staggered $2$-partition for ${\bf s}=(0,0)$.

\end{exa}

\section{ Uglov multipartitions and crystal approach}
The set of $e$-regular partitions that we have seen above is in fact a particular case of the set of Uglov $l$-partitions. 
These objects naturally appear in the context of the crystal graph theory, canonical bases theory and the representation theory of Ariki-Koike algebras. We here begin this section  with their definitions.

\subsection{Uglov multipartitions}\label{ugsec}

\begin{abs}\label{ra}
Le ${\bf s}=(s_1,\ldots,s_l)\in \mathbb{Z}^l$, $e\in \mathbb{Z}_{>1}$ and 
 $\mathfrak{j}\in \mathbb{Z}/e\mathbb{Z}$. 
Let  ${\boldsymbol{\lambda}}$ be an  $l$-partition. We can consider its set of addable and
removable $\mathfrak{j}$-nodes. Let $w_{\mathfrak{j}}(\ulambda)$ be the word obtained first by writing the
addable and removable $\mathfrak{j}$-nodes of ${\boldsymbol{\lambda}}$ in {increasing}
order with respect to $<$ (see \S \ref{refor})
and then by encoding each addable $\mathfrak{j}$-node by the letter $A$ and each removable $\mathfrak{j}$-node by the letter $R$. If we delete all the subword of the form $RA$ in this word,
 we obtain again a word in $R$ and $A$ which we write by 
$\overline{w}_{\mathfrak{j}}(\ulambda)$ (this word will have a role to play in the sequel). 
Now delete in this word again all the factors $RA$ and continue until you reach a word
 with no such factors.
Write $\widetilde{w}_{\mathfrak{j}}(\ulambda)=A^{p}R^q$ for the
word derived from $w_{\mathfrak{j}} (\ulambda)$. Thus 
$\widetilde{w}_{\mathfrak{j}}(\ulambda)$ is a subword of 
$\overline{w}_{\mathfrak{j}}(\ulambda)$ which is itself a subword of 
$w_{\mathfrak{j}} (\ulambda)$.

The removable $\mathfrak{j}$-nodes in  $\widetilde{w}_{\mathfrak{j}}(\ulambda)$ are called the 
 {\it normal removable} $\mathfrak{j}$-nodes. If $r>0$, the leftmost removable $\mathfrak{j}$-node in 
 $\widetilde{w}_{\mathfrak{j}}$ is called the {\it good removable $\mathfrak{j}$-node}.
 
The addable $\mathfrak{j}$-nodes in  $\widetilde{w}_{\mathfrak{j}}(\ulambda)$ are called the 
 {\it normal addable} $\mathfrak{j}$-nodes. If $r>0$, the rightmost addable $\mathfrak{j}$-node in 
 $\widetilde{w}_{\mathfrak{j}}$ is called the {\it good addable $\mathfrak{j}$-node}. 
 
\end{abs}

\begin{Def}\label{uglov}
  The set of {\it Uglov $l$-partitions} $\Uglov{e}{\bf s}$ is defined   recursively  as follows.
\begin{itemize}
   \item We have $\uemptyset:=(\emptyset,\emptyset,...,\emptyset)\in{ \Uglov{e}{\bf s}}$.
    \item If $\ulambda\in\Uglov{e}{\bf s}$, there exist $\mathfrak{j}\in \mathbb{Z}/e\mathbb{Z}$ and a good removable $\mathfrak{j}$-node $\gamma$ such that if we remove $\gamma$ from  $\ulambda$, the resulting  $l$-partition is in $\Uglov{e}{\bf s}$.
\end{itemize}
We also set $\Uglov{e}{\bf s}(n)=\Uglov{e}{\bf s}\cap \Pi^{l}(n)$.

\end{Def}

Of course, if $\ulambda\in\Uglov{e}{\bf s}$ and  there exist $\mathfrak{j}\in \mathbb{Z}/e\mathbb{Z}$ and a good addable  $\mathfrak{j}$-node $\gamma$ then the $l$-partition $\umu$ such that $\mathcal{Y} (\umu)=\mathcal{Y} (\ulambda)\cup\{\gamma\}$ is in $\Uglov{e}{\bf s}$.

\begin{Rem}\label{ere}
 It is easy to see that for $l=1$, the set $\Uglov{e}{\bf s}$ always correspond to the set of $e$-regular partitions. We will see in the next ssubsection anoother particular cases of Uglov $l$-partitions. 
\end{Rem}

\begin{Rem}\label{translation}
Let ${\bf s}=(s_1,\ldots,s_l)\in \mathbb{Z}^l$. From the definition, it is easy to see  that for all 
 $k\in \mathbb{Z}$, if we denote ${\bf s}+k=(s_1+k,\ldots,s_l+k)$,  we have $\Uglov{e}{\bf s}=\Uglov{e}{{\bf s}+k}$.

\end{Rem}
\begin{Rem}\label{infty}
In the following, we only consider the case where $e\in \mathbb{Z}_{>1}$. However, the case $e=\infty$ follows by taking $e>>0$. 

\end{Rem}

Fix $n\in \mathbb{Z}_{>0}$ and  assume, as in \ref{kl},  that the multicharge ${\bf s}=(s_1,\ldots,s_l)\in \mathbb{Z}^l$ is asymptotic, that is;
\begin{equation*}\label{hs} 
 {\bf (H)}\qquad \forall i\in \{2,\ldots,l\},\ s_{i}-s_{i-1}\geq n-1+e.
 \end{equation*}
 The associated Uglov $l$-partitions are known as {\it Kleshchev multipartitions}.

\subsection{FLOTW multipartitions}
Assume that $e\in \mathbb{Z}_{>0}$ and $l\in \mathbb{Z}_{>0}$ and define:
 $$\mathcal{S}_e^l:=\{(s_1,\ldots,s_l)\in \mathbb{Z}^l\ |\ \forall 1\leq i<j\leq l,\ 0\leq s_{i}-s_j\leq  e\},$$ 
In the case where ${\bf s}=(s_1,\ldots,s_l)\in \mathcal{S}_e^l$,  the set of Uglov multipartitions have a nice alternative description (see for example \cite[Th. 6.3.2]{GJ} for its proof.) 
 
\begin{Prop}[Foda-Leclerc-Okado-Thibon-Welsh]\label{flotwdef}
Assume that ${\bf s}=(s_1,\ldots,s_l)\in \mathcal{S}_e^l$.  The set $\Uglov{e}{\bf s}$ of Uglov $l$-partitions is the  set of $l$-partitions  $\ulambda={(\lambda^{1} ,...,\lambda^{l})}$ such that:
\begin{enumerate}
\item for all $1\leq{j}\leq{l-1}$ and $i\in \mathbb{Z}_{>0}$, we have:
\begin{align*}
&\lambda_i^{j}\geq{\lambda^{j+1}_{i+s_{j+1}-s_{j}}},\\
&\lambda^{l}_i\geq{\lambda^{1}_{i+e+s_1-s_{l}}};
\end{align*}
\item  for all  $k>0$, among the residues of the nodes of the vertical boundary of the form $(a,\lambda^c_a,c)$
 with $a\in \mathbb{Z}_{>0}$, $c\in \{1,\ldots,l\}$ and $\lambda^c_a=k$, at least one element of  $\{0,1,...,e-1\}$ does not occur.
\end{enumerate}
 Such $l$-partitions are called FLOTW $l$-partitions.
\end{Prop}

\begin{Rem}
For $l=1$, the $1$-partitions may be naturally identified with the partitions. The set of FLOTW $1$-partitions are then   identified with the set of $e$-regular partitions.  This is consistent with Remark \ref{ere}. 
\end{Rem}
\begin{Rem}\label{fla}
Note that for ${\bf s}=(s_1,\ldots,s_l)\in \mathcal{S}_e^l$ and for all $c\in\{1,\ldots,l\}$, we have
  $\ulambda={(\lambda^{1} ,...,\lambda^{l})}\in    \Uglov{e}{\bf s}$ if and only if
   $\ulambda={(\lambda^{c+1} ,...,\lambda^{l},\lambda^1,\ldots,\lambda^c)}\in    \Uglov{e}{{\bf s}'}$
    for ${\bf s}'=(s_{c+1},\ldots,s_l,s_1+e,\ldots,s_c+e)$  which is also in  
$\mathcal{S}_e^l$ 
\end{Rem}

\section{The main result}

We here state our main result and expose the strategy to prove it.

\subsection{The main theorem}

One of the main result of this paper is the following. 


\begin{Th}\label{main} Assume that $e\in \mathbb{Z}_{>1}$ and ${\bf s}=(s_1,\ldots,s_l)\in \mathbb{Z}^l$. 
An $l$-partition is an Uglov  $l$-partition with respect to $(e,{\bf s})$ if and only if it is a staggered $l$-partition with respect to $(e,{\bf s})$. 
\end{Th}

\begin{Rem}
Assume that $e\in \mathbb{Z}_{>1}$ and ${\bf s}=(s_1,\ldots,s_l)\in \mathbb{Z}^l$  are  such that 
 for all $i=2,\ldots,l$, we have $s_{i}-s_{i-1}\geq n-1+e$. Then the set of Uglov $l$-partitions correspond 
  to the set of Kleshchev $l$-partitions. We thus obtain a new characterization of this set of $l$-partitions which does not use the notion of crystal graph. 

\end{Rem}

\begin{Rem}\label{conven}
There are different conventions for the use of this notion of Kleshchev $l$-partitions. In this paper, we use the conventions of \cite{GJ} (which come from the works of Uglov \cite{Ug}). The papers \cite{BK,DJM,Hu} use another convention. In particular, the order on nodes used in these papers is the reversed order of ours. As a consequence, the Kleshchev $l$-partitions in our paper  correspond to the conjugate of the Kleshchev $l$-partitions obtained in {\it op. cit.} (the conjugate of an $l$-partition $(\lambda^1,\ldots,\lambda^l)$ is $({\lambda^l} ',\ldots,{\lambda^1} ')$ where $\lambda'$ denotes the usual conjugate, or transpose, of the partition $\lambda$). 
 Of course, it is straightforward to translate our result in terms of the other convention.

\end{Rem}

\subsection{The  strategy of the proof}

Let us end this section by explaining  the strategy to prove the above theorem:
\begin{itemize}
  \item In section \ref{flotw}, we will consider a particular case: the case where ${\bf s}\in \mathcal{S}_e^l$.   We show that, in this case, 
   the Uglov $l$-partitions (also known as FLOTW $l$-partitions) are staggered $l$-partitions (and reciprocally). In fact, we will even show a stronger result that this one.
\item We then recall and develop previous results on crystal isomorphisms in \S \ref{isosec}. We show that all classes of Uglov $l$-partitions are in bijection with the Kleshchev $l$-partitions and the FLOTW $l$-partitions. We moreover describe and study these bijections.
\item We use these bijections and the previous results on FLOTW multipartitions  to deduce our main result in the ninth part.

\end{itemize}

\section{The case of FLOTW multipartitions}\label{flotw}

The aim of this section  is to show that in the case where  ${\bf s}=(s_1,\ldots,s_l)\in \mathcal{S}_e^l$, 
the associated Uglov $l$-partitions are always staggered $l$-partitions. 
It is a remarkable fact that the 
 Uglov $l$-partitions have a simple non recursive definition in this case. 
 
 \subsection{Definition and first properties}

In the following, we will need two important lemmas. The proof of the first one can be found in  \cite[Lemma 5.7.17]{GJ}. In this reference, 
  the result concerns only removable and addable nodes but one can check that if we replace "removable node" by "extended node of the vertical boundary", the proof is still correct. One can also check that only the first axiom $(1)$ in the definition of FLOTW $l$-partition is used in the proof. 

\begin{lemma}\label{L1}
Let  ${\bf s}=(s_1,\ldots,s_l)\in \mathcal{S}_e^l$ and  $\ulambda$ be an $l$-partition satisfying $(1)$ in Proposition \ref{flotwdef}. Let  $\mathfrak{j}\in \mathbb{Z}/e\mathbb{Z}$, let $\gamma_1=(a_1,b_1,c_1)\in \mathcal{E}_{\mathfrak{j}} (\ulambda)$, 
and let $\gamma_2=(a_2,b_2,c_2)\in \mathcal{E}_{\mathfrak{j}} (\ulambda)$.  
 Then if $\lambda_{a_1}^{c_1}<\lambda_{a_2}^{c_2}$, we have
  $\gamma_1<\gamma_2$.
 
\end{lemma}

The second lemma concerns nodes with possible distinct residues:

\begin{lemma}\label{LL1}
Let  ${\bf s}=(s_1,\ldots,s_l)\in \mathcal{S}_e^l$ and  $\ulambda$ be an $l$-partition satisfying $(1)$ in Proposition \ref{flotwdef}. Let  $\gamma_1=(a_1,b_1,c_1)$ be a removable node for $\ulambda$ 
and let $\gamma_2=(a_2,b_2,c_2)$ be a $\mathfrak{j}$-node of $\ulambda$ in the extended Young diagram. Assume that $\lambda_{a_1}^{c_1}=\lambda_{a_2}^{c_2}$ 
 then we have 
 \begin{enumerate}
\item  $\operatorname{cont} (\gamma_2)\geq \operatorname{cont} (\gamma_1)$ if $c_1<c_2$
\item  $\operatorname{cont} (\gamma_2)\geq \operatorname{cont} (\gamma_1)-e$ if $c_1>c_2$ 
\end{enumerate}
\end{lemma}
\begin{proof}
Assume first that $c_1< c_2$, by hypothesis, as $\gamma_1$ is removable, we have $\lambda_{a_1}^{c_1}>\lambda_{a_1+1}^{c_1}\geq \lambda^{c_2}_{a_1+1+s_{c_2}-s_{c_1}}$. This implies that $a_2\leq a_1+s_{c_2}-s_{c_1}$ and thus that 
 $\operatorname{cont} (\gamma_1)\geq \operatorname{cont} (\gamma_2)$.
 Assume that $c_1>c_2$ then we consider the multipartition $\ulambda={(\lambda^{c_1} ,...,\lambda^{l},\lambda^1,\ldots,\lambda^{c_1-1})}\in    \Uglov{e}{{\bf s}'}$
    for ${\bf s}'=(s_{c_1},\ldots,s_l,s_1+e,\ldots,s_{c_1-1}+e)$  which is also in  
$\mathcal{S}_e^l$ by  Remark \ref{fla} and we can apply the first case to deduce  the result.

\end{proof}

\begin{Rem}\label{LLrem}
From the above result, we deduce that if 
 ${\bf s}=(s_1,\ldots,s_l)\in \mathcal{S}_e^l$ and if  $\ulambda$ satisfies condition $(1)$ 
  in Proposition \ref{flotwdef}, 
 if $\gamma_1=(a_1,b_1,c_1)$ and 
 $\gamma_2=(a_2,b_2,c_2)$
  are two nodes of the same residue  in the vertical boundary with $b_1=b_2$ then we have 
  $|\textrm{cont}(\gamma_1)-\textrm{cont}(\gamma_2)|\leq e$. 
   It follows from the fact that all the components of $\ulambda$ are $e$-regular and thus that 
    there exists $a_1'<a+e$ such that $\lambda_{a_1'}^{c_1}=b_1$ and $(a_1',b_1,c_1)$ is removable.

\end{Rem}

\subsection{Combinatorial study of FLOTW multipartitions}

We now need a series of results which will help us to obtain  a generalization of our main theorem in the case of FLOTW $l$-partitions.  
First, the following proposition  studies what happens if a multipartition satisfies the first condition without satisfying the second condition of FLOTW multipartition.

\begin{Prop}\label{period}
Let  ${\bf s}=(s_1,\ldots,s_l)\in \mathcal{S}_e^l$ and  $\ulambda$ be an $l$-partition satisfying $(1)$ in Proposition \ref{flotwdef}.  Assume that condition $(2)$ in Proposition \ref{flotwdef} is not satisfied, then there exists a sequence of extended nodes $(\gamma_1,\ldots,\gamma_e)$ in the vertical boundary of the Young diagram of $\ulambda$ such that for all $i=1,\ldots,e-1$, we have $\operatorname{cont}(\gamma_{i+1})=\operatorname{cont} (\gamma_i)+1$ and $c_{i+1}\geq c_i$.
\end{Prop}
\begin{proof}
By hypothesis, there exists $k> 0$ such that the set of residues of the nodes of the form
 $(a,k,c)$ in the vertical boundary of  the extended Young diagram of $\ulambda$ is $\mathbb{Z}/e\mathbb{Z}$. Let us denote by $\mathcal{N}$ the set of all such nodes. Among the nodes with minimal content $j$ in $\mathcal{N}$, we let $\gamma_1=(a_1,k,c_1)$ to be the node with component $c_1$ minimal. If $(a_1-1,k,c_1)$ is in $\mathcal{N}$ then set $\gamma_2=(a_1-1,k,c_1)$ and consider $(a_1-2,k,c_1)$ etc. Let us assume that 
  $(a_1-r_1,k,c_1)$ is not in $\mathcal{N}$ so that the nodes $\gamma_i=(a_i-(i-1),k,c_1)$ with $i=1,\ldots,r_1$ are already defined. We have $\operatorname{cont} (\gamma_{r_1})=j+r_1-1$ and we remove the nodes $\gamma_1$, \ldots, $\gamma_{r_1}$ from $\mathcal{N}$. 
  
  In $\mathcal{N}$, we can assume that we have a node $\eta=(a_2,k,c_2)$ with residue $j+r_1 (\textrm{mod }e)$ (otherwise we have already found our sequence).  In addition,  we can assume that the content of $\eta$ is minimal, and that  $c_2$ is minimal among the nodes with this minimal content. First note that 
  we must have $c_2>c_1$. This follows from Lemma \ref{LL1} together with the fact that $j$ is minimal.
  
   Now  the content of $\eta$ is of the form $j+r_1+t.e$ with $t\in \mathbb{Z}$. By the minimality of $j$, we have $t\geq 0$.  If $t>1$ then by Lemma \ref{LL1}, we must have a removable node  of the form 
   $(a',k,c_2)$ with content less than $j+e$. This implies that the set of residues of the  nodes of the form
   $(a',k,c_2)$ is $\{0,1,\ldots,e-1\}$ and the desired sequence is given by a sequence formed by 
   this kind of nodes. 
  Let us now assume that $t=1$, Lemma \ref{LL1} implies that we can assume that we have a removable node of the form 
   $(a',k,c)$  with content $j+e$. We obtain $\lambda_{a_1}^{c_1}-a_1+s_{c_1}=\lambda^{c'}_{a'}-a'+s_{c'}-e$ and thus 
   $a_1=a'+s_{c_1}-s_{c'}+e$. We know that  $\lambda_{a'-{r_1}}^{c'}=k$ and that $\lambda^{c'}_{a'-r_1}\geq \lambda^{c_1}_{a'-r_1+s_{c_1}-s_{c'}+e}$. This implies $\lambda^{c_1}_{a_1-r_1}\leq k$ which contradicts the fact that    $(a_1-r_1,k,c_1)$ is not in $\mathcal{N}$. 
 
 So we have that $t=0$, we then set $\gamma_{r_1+1}=\eta$. If $(a_2-1,k,c_2)$ is in $\mathcal{N}$ then set $\gamma_{r_1+2}=(a_2-1,k,c_2)$ and consider $(a_2-2,k,c_2)$ etc. Let us assume that 
  $(a_2-r_2,k,c_2)$ is not in $\mathcal{N}$ so that the nodes $\gamma_{i+r_1}=(a_2-(i-1),k,c_2)$ with $i=1,\ldots,r_2$ are already defined. We continue by considering a node of residue $j+r_1 +r_2(\textrm{mod }e)$ 
 and this process ends when we consider a node of residue $j+e-1 (\textrm{mod  }e)$.

\end{proof}
\begin{Rem}
The above sequence of nodes give the existence of a period as defined in \cite[\S 2.3]{JL2}.
 This notion plays  a crucial role in the combinatorics of crystals in affine type $A$ (see \cite{JL2, Ge}).

\end{Rem}

\begin{exa}
  Let $l=3$ and take the $3$-partition $\ulambda=(4.3.3,3.2,3.1)$ with ${\bf s}=(0,0,1)$ and $e=4$. Then $\ulambda$ satisfies the hypotheses of the proposition and the  Young diagram 
   with content is:
  
\vspace{0,5cm}  
  
  \centerline{\small{
 \Bigg(\  \ \ \ \ytableausetup
{mathmode}\begin{ytableau}
  { 0} & { 1} & { 2} & 3\\
 \overline{1} & 0 & *(lightgray) { 1} \\
 \overline{2} & \overline{1} & *(lightgray){ 0} \\
\end{ytableau}\ ,
\begin{ytableau}
  { 0} & { 1} & *(lightgray){ 2} \\
\overline{1}& 0 \\
\end{ytableau}\ ,
\begin{ytableau}
  { 1} & { 2} & *(lightgray){ 3} \\
 0 \\
\end{ytableau}\  \ \ \  \Bigg)}}

\vspace{0,5cm}  

We see the sequence of the desired nodes of the proposition in gray in the above diagram. 
\end{exa}

\begin{lemma}\label{ll1}
Let  ${\bf s}=(s_1,\ldots,s_l)\in \mathcal{S}_e^l$ and let $\ulambda\in \Uglov{e}{\bf s}$. We assume that there exist $c\in \{1,\ldots,l\}$ and $i\in \mathbb{Z}_{>0}$ such that 
$$\lambda^c_i-1=\lambda^c_{i+1}=\ldots=\lambda^c_{i+e-1}\neq 0.$$
Denote $\gamma_c:=(i,\lambda^c_i,c)$ and $\mathfrak{j}:=\operatorname{res} (\gamma_c)$.   Then the node $\eta:=(i+e-1,\lambda^c_{i+e-1},c)$ is a removable $\mathfrak{j}$-node and for each $m\in \{1,\ldots,l\}\setminus \{c\}$, there exists  an extended  $\mathfrak{j}$-node of the horizontal boundary  $\gamma_m$ of the partition $\lambda^m$ for $m\in \{1,\ldots,l\}\setminus \{c\}$   such that:
$$\gamma_c >_{co} \gamma_{c+1} >_{co}  \ldots >_{co} \gamma_l >_{co} \gamma_1 >_{co} \ldots >_{co} \gamma_{c-1}>_{co} \eta.$$
\end{lemma}
\begin{proof}
First, note that $\eta$ must be removable otherwise $\ulambda$ does not satisfied the second condition to be a FLOTW $l$-partition and it is clear that its residue is the same as the one of $\gamma$.  Let us denote by ${j}\in \mathbb{Z}$ the content of $\gamma_c$. We will see that for all $c'>c$ there is an extended  node $\gamma_{c'}$ in the horizontal boundary of $\lambda^{c'}$ with content $j$ and for all $c'<c$ there is an extended node $\gamma_{c'}$ in the horizontal boundary of $\lambda^{c'}$ with content $j-e$. We then obtain:
$$\gamma_c >_{co} \gamma_{c+1} >_{co} \ldots \gamma_l >_{co} \gamma_1 >_{co} \ldots >_{co} \gamma_{c-1}>_{co} \eta.$$
which is what we want. Let $c'\in \{1,\ldots,l\}\setminus \{c\}$.

    \begin{itemize}
    \item Assume that $c'>c$. By hypothesis, we have $\lambda^{c'}_{i+s_{c'}-s_c-1}\geq \lambda^{c}_{i+e-1}$.
     As we know that  
     $\lambda^c_i-1=\lambda^{c}_{i+e-1}$, we obtain that  $(i+s_{c'}-s_c-1,\lambda^{c}_{i+e-1},c')$ is a node of $\ulambda$ and its content is $j$. Moreover, we have $\lambda^c_i\geq \lambda^{c'}_{i+s_{c'}-s_c}$. If 
      $\lambda^c_i= \lambda^{c'}_{i+s_{c'}-s_c}$ then the content of the node $(i+s_{c'}-s_c,\lambda^{c}_{i+s_{c'}-s_c},c')$ is $j$ and as $\lambda^c_{i+1}\geq \lambda^{c'}_{i+s_{c'}-s_c+1}$, it is removable. 
      We can thus set $\gamma_{c'}=(i+s_{c'}-s_c,\lambda^{c}_{i+s_{c'}-s_c},c')$. 
       Otherwise $\lambda^c_i> \lambda^{c'}_{i+s_{c'}-s_c}$  and thus  
      $\lambda^c_{i+e-1}\geq  \lambda^{c'}_{i+s_{c'}-s_c}$. Now we cannot have 
     $\lambda^c_{i+e-1}=  \lambda^{c'}_{i+s_{c'}-s_c}$: the residue of the node $(i+s_{c'}-s_c,
     \lambda^{c'}_{i+s_{c'}-s_c},c')$ would be $\mathfrak{j}-1$ and 
   the second condition of FLOTW $l$-partitions would be  violated. Thus, we get  $\lambda^c_{i+e-1}>  \lambda^{c'}_{i+s_{c'}-s_c}$
  and    $(i+s_{c'}-s_c-1,\lambda^{c}_{i+e-1},c')$ is in the horizontal boundary. We can thus set 
  $\gamma_{c'}=(i+s_{c'}-s_c-1,\lambda^{c}_{i+e-1},c')$.

    \item The case  $c'<c$ follows from the above case together with Remark \ref{fla}. 
    \end{itemize}

\end{proof}

The following three lemmas define for a FLOTW multipartition and two associated removable nodes, the notion of {\it $(1)$-connected nodes} and {\it $(2)$-connected nodes} which will be used in the next subsections.

\begin{lemma}\label{l201}
Let  ${\bf s}=(s_1,\ldots,s_l)\in \mathcal{S}_e^l$ and let $\ulambda\in \Uglov{e}{\bf s}$. We assume that there exist $c\in \{1,\ldots,l\}$ and $i\in \mathbb{Z}_{>0}$ such that $\gamma_c:=(i,\lambda^c_i,c)$ is a removable $\mathfrak{j}$-node for some $\mathfrak{j}\in \mathbb{Z}/e\mathbb{Z}$. Assume that $\lambda_i^c>1$ and that  the set of residues of the nodes appearing in  the vertical boundary of 
parts of length $\lambda_i^c-1$ is $\{0,1,\ldots,e-1\}\setminus \{\mathfrak{j}-1\}$. Then there exists a removable $\mathfrak{j}$-node 
  in the vertical boundary of a part with length  $\lambda_i^c-1$. A removable node satisfying this condition is said to be $(1)$-connected with $\gamma_c$.  
\end{lemma}
\begin{proof}
We denote by $j$ the content of $\gamma$. By hypothesis, there exists a $\mathfrak{j}$-node  $\eta=(i',\lambda^{c'}_{i'},c')$ such that $\lambda^{c'}_{i'}=\lambda^c_i-1$. This must be a removable node otherwise the set of residues of the nodes in the vertical boundary of 
parts with length $\lambda_i^c-1$ is $\{0,1,\ldots,e-1\}$. 
\end{proof}

\begin{lemma}\label{l2}
Let  ${\bf s}=(s_1,\ldots,s_l)\in \mathcal{S}_e^l$ and let $\ulambda\in \Uglov{e}{\bf s}$. We assume that there exist $c\in \{1,\ldots,l\}$ and $i\in \mathbb{Z}_{>0}$ such that $\gamma_c:=(i,\lambda^c_i,c)$ is a removable $\mathfrak{j}$-node for some $\mathfrak{j}\in \mathbb{Z}/e\mathbb{Z}$. Assume that the set of residues of the nodes in the vertical boundary of 
parts of length $\lambda_i^c-1$ is $\{0,1,\ldots,e-1\}\setminus \{\mathfrak{j}-1\}$. Then  there is a sequence of consecutive extended $\mathfrak{j}$-nodes of the horizontal boundary between every $(1)$-connected removable node $\eta$ with $\gamma_c$ and the node $\gamma_c$. 
\end{lemma}
\begin{proof}

Let $j\in \mathbb{Z}$ be the content of $\gamma_c$. We set $\eta=(i',\lambda^{c'}_{i'},c')$. Assume first that $c'=c$ then we can use lemma \ref{ll1} to deduce the result. So  
let us assume  that $c'>c$. We have $\lambda_{i}^c-1\geq  \lambda_{i+1}^c> \lambda^c_{i+e}\geq \lambda^{c'}_{i+e+s_{c'}-s_c}$ because $\lambda^c$ is $e$-regular and because of the definition of the FLOTW $l$-partitions. We obtain $\lambda^{c'}_{i'}\geq \lambda^{c'}_{i+e+s_{c'}-s_c}+1$ from what we deduce $i'< i+e+s_{c'}-s_c$. Now we have $\lambda^{c'}_{i'}-i'+s_{c'}\geq \lambda^c_i-1-(i+e+s_{c'}-s_c-1)+s_{c'}$ thus the content of $\eta$ is $j$ or $j-e$.

Let us first assume that this content is $j$. We then have $\lambda^{c'}_{i'}-i'+s_{c'}=\lambda^c_i-1-i'+s_{c'}$ from what we deduce $i'=i+s_{c'}-s_c-1$. Assume now that we have an integer $c''$ such that $c<c''<c'$.
 We have $\lambda^{c''}_{i+s_{c''}-s_c-1}\geq \lambda^{c'}_{i+s_{c'}-s_c-1}$ and thus 
 $\lambda^{c''}_{i+s_{c''}-s_c-1}\geq \lambda^{c'}_{i'}$. Note that the content of the node
  $(i+s_{c''}-s_c-1,\lambda^{c'}_{i'},c'')$ is $j$ (if $i+s_{c''}-s_c-1=0$ then we can see it as an extended node on the virtual horizontal boundary).  Now, we also have 
  $\lambda^{c}_i\geq \lambda^{c''}_{i+s_{c''}-s_c}$ and we have three cases to consider:
\begin{itemize}
\item If  $\lambda^{c}_i= \lambda^{c''}_{i+s_{c''}-s_c}$ then the node $(i+s_{c''}-s_c,\lambda^c_i,c'')$ has content $j$ and it is removable as $\lambda^{c}_{i+1}\geq \lambda^{c''}_{i+s_{c''}-s_c+1}$ and $\gamma$ is removable thus it is on the horizontal boundary.
\item If $\lambda^{c}_i= \lambda^{c''}_{i+s_{c''}-s_c}+1$ the node 
$(i+s_{c''}-s_c,\lambda^{c''}_{i+s_{c''}-s_c},c'')$ has residue $\mathfrak{j}-1$ on a part with length $\lambda^c_i-1$ and this contradicts our hypotheses. 
\item Otherwise  $\lambda^{c'}_{i'}> \lambda^{c''}_{i+s_{c''}-s_c}$  and  $(i+s_{c''}-s_c-1,\lambda^{c'}_{i'},c'')$ is in the horizontal boundary. 
\end{itemize}

 Let us now assume that the content of $\eta$ is  
  $j-e$.  This means that we have $i'=i+e+s_{c'}-s_c-1$ and thus, by the first condition of FLOTW $l$-partition,we get $\lambda^c_{i+e-1}\geq \lambda^{c'}_{i'}$ and thus 
  $\lambda^c_{i+e-1}=\lambda^c_i-1$. 
  
   Now we can use Lemma \ref{ll1} to deduce that 
 there is a sequence of $l$ removable nodes of the horizontal boundary between $\gamma$ and 
  $(i+e-1,  \lambda^c_{i+e-1},c)$ which is of content $j-e$ (and removable). We can argue then exactly as above to see that 
   there is a sequence of  consecutive extended nodes of the horizontal boundary with content $j-e$ 
between   $(i+e-1,  \lambda^c_{i+e-1},c)$ and $\eta$. 

Finally, note that, as usual,  the case where $c'>c$ follows again from the above case together with Remark \ref{fla}.

\end{proof}

\begin{lemma}\label{l20}
Let  ${\bf s}=(s_1,\ldots,s_l)\in \mathcal{S}_e^l$ and let $\ulambda\in \Uglov{e}{\bf s}$. We assume that there exist $c\in \{1,\ldots,l\}$ and $i\in \mathbb{Z}_{>0}$ such that $\gamma_c:=(i,\lambda^c_i,c)$ is a removable $\mathfrak{j}$-node for some $\mathfrak{j}\in \mathbb{Z}/e\mathbb{Z}$. Let $c'\in \{1,\ldots,l\}$ with $c'\neq c$ and set $j=i+s_{c'}-s_c$ if $c'>c$ and $j=i+s_{c'}-s_c+e$ if $c'<c$. Assume that 
 $\lambda^c_i=\lambda^{c'}_j$ then $\gamma_{c'}=(j,\lambda^{c'}_j,c')$ is a removable $\mathfrak{j}$-node. 
 A removable node satisfying this condition is said to be $(2)$-connected with $\gamma_{c'}$. If $\gamma_{c'}$ is such a a node then  
there exists a series of consecutive removable $\mathfrak{j}$-nodes between $\gamma_c$ and $\gamma_{c'}$ which are $(2)$-connected one with each other. 
   
\end{lemma}
\begin{proof}
 Let $c\in \{1,\ldots,l\}$ and let $i\in \mathbb{Z}_{>0}$. Let us first assume that $c'>c$. By hypothesis, we have $\lambda^c_i\geq \lambda^{c'}_{i+s_{c'}-s_c}$. Assume  that
  $\gamma_c=(i,\lambda^c_i,c)$ is a removable node   and $\lambda^c_i= \lambda^{c'}_{i+s_{c'}-s_c}$
then the node $\eta=(i+s_{c+1}-s_c,  \lambda^{c+1}_{i+s_{c+1}-s_c},c+1)$ must be removable (otherwise we would have
$\lambda^c_{i+1}< \lambda^{c+1}_{i+s_{c+1}-s_c+1}$) 
and  it has the same residue as $\gamma_c$. Now if $c<c''<c'$, then we must have $\lambda^{c}_i= \lambda^{c''}_{i+s_{c''}-s_{c'}}$. The $\mathfrak{j}$-node  $\gamma=(i+s_{c''}-s_c,\lambda^c_{i+s_{c''}-s_c},c'')$ 
is thus removable and $(2)$-connected to $\gamma_c$, the result follows. 
If $c'<c$ we again conclude using the above reasoning together with Remark \ref{fla}. 

\end{proof}

\subsection{Nature of a node}

We here define the nature of a node in the extended Young diagram and then finish the subsection with a proposition which shows the existence of a staggered sequence for every FLOTW $l$-partitions.

\begin{abs}\label{exte}
In this paragraph,  we go back to the general case for the moment with ${\bf s}\in \mathbb{Z}^l$.

 Let us assume that $\ulambda \in \Uglov{e}{\bf s}$ and let $\mathfrak{j}\in \mathbb{Z}/e\mathbb{Z}$. Let $j\in \mathbb{Z}$ in the class of $\mathfrak{j}$ module $e\mathbb{Z}$. We have already remarked that, given a component $\lambda^{c}$ of $\ulambda$, there is a unique extended node $\gamma$ with content $j$ which is addable  or on the boundary of $\ulambda$. Recall that we denote in this case $\gamma\in\mathcal{E}_{\mathfrak{j}} (\ulambda)$. Let $v_{\mathfrak{j}}(\ulambda)$ be the (infinite) word obtained by reading the element in $\mathcal{E}_{\mathfrak{j}} (\ulambda)$ with residue $\mathfrak{j}$ in increasing order with respect to $>$ and by writing:
\begin{itemize}
\item $A$ if the associated extended node is addable,
\item $R$ if the associated extended node is removable
\item $B_{hor}$ if the associated extended node is in the horizontal boundary of $\ulambda$ without being removable,
\item $B_{ver}$ if the associated extended node is in the vertical boundary without being removable.  
\end{itemize}
Each (addable or extended) node is indeed in one of the above category which is called the {\it nature} of the (addable or extended)   node.  
We can note that consecutive tletters  of this word are associated to  elements of $\mathcal{E}_{\mathfrak{j}} (\ulambda)$ which are consecutive with respect to $<$.  
In addition,  the word $v_{\mathfrak{j}}(\ulambda)$  is always of the form 
$Y_1YY_2$ where $Y_1$ is an infinite  sequence of $B_{ver}$, $Y_2$ an infinite sequence of $B_{hor}$, 
 and $Y$ a finite sequence.  Note also that if we remove all the (addable or extended)  nodes of nature 
$B_{ver}$ or $B_{hor}$ in this sequence, we obtain the word $w_{\mathfrak{j}}(\ulambda)$ (defined in \S \ref{ugsec}) which is thus a subword of $v_{\mathfrak{j}}(\ulambda)$. The word 
$\widetilde{w}_{\mathfrak{j}}(\ulambda)$ is also a subword for $v_{\mathfrak{j}}(\ulambda)$.

\begin{exa}
For example take $e=4$ ans ${\bf s}=(0,0,2,1,0)$ and the $5$-partition
 $\ulambda=(2,22,2,3.1,\emptyset)$:
 
 \vspace{0,5cm}

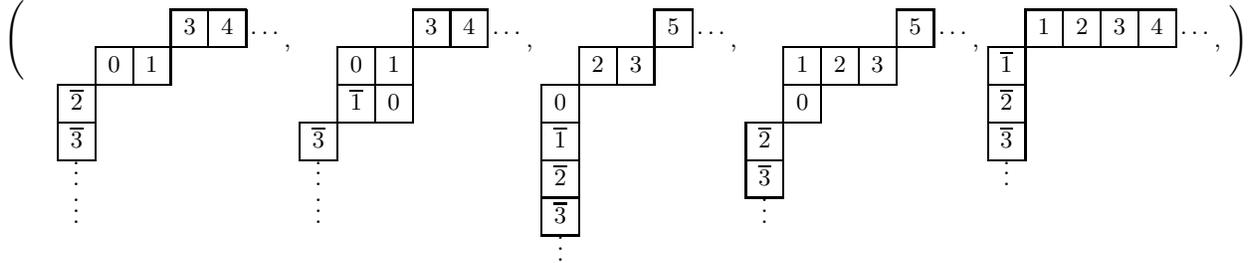
  \centerline{\small{
 \Bigg(\  \ \ \ \ytableausetup
{mathmode}\begin{ytableau}
\none  &  \none  & \none  & 3  &  4&  \none[\dots] \\
\none  &  { 0} & { 1}  \\
\overline{2}  \\
\overline{3}  \\  
\none[\vdots] \\
 \none[\vdots] \\
\end{ytableau},
\begin{ytableau}
\none  &  \none  & \none  & 3  &  4&  \none[\dots] \\
\none  &  { 0} & { 1}  \\
\none & \overline{1} & 0  \\
\overline{3}  \\  
\none[\vdots] \\
 \none[\vdots] \\
\end{ytableau},
\begin{ytableau}
\none  &  \none  &\none   & 5  &    \none[\dots] \\
\none  &  { 2} & { 3}  \\
0    \\
\overline{1}  \\
\overline{2}  \\
\overline{3}  \\  
 \none[\vdots] \\
\end{ytableau},
\begin{ytableau}
\none  &  \none  &  \none &\none & 5  &   \none[\dots] \\
\none  & 1&  { 2} & { 3}  \\
\none & 0  \\  
\overline{2}  \\
\overline{3}  \\
 \none[\vdots] \\
\end{ytableau}, 
\begin{ytableau}
\none  &  1&2  &  3 &4 &    \none[\dots] \\
\overline{1}    \\ 
\overline{2}  \\
\overline{3}  \\
 \none[\vdots] \\
\end{ytableau}, 
 \Bigg)}}

 \noindent One can see that 
 $$v_1(\ulambda)=Y_1  B_{hor}AAB_{ver}R.Y_2.$$

\end{exa}
\end{abs}

\begin{Prop}\label{princ0}
Let  ${\bf s}=(s_1,\ldots,s_l)\in \mathcal{S}_e^l$ and assume that $\ulambda\in\Uglov{e}{\bf s}$.
 Let $\mathfrak{j}\in \mathbb{Z}/e\mathbb{Z}$. 
 Assume that $v_{\mathfrak{j}} (\ulambda)=X_1X_2$ where $X_1$ and $X_2$ are two words in 
 $\{A,R,B_{ver},B_{hor}\}$ satisfying the following condition: no node associated to a letter $R$ 
  in $X_1$ is $(1)$ or $(2)$-connected with a node   associated to a letter $R$ 
  in $X_2$.  Then, 
  if we remove all the removable  $\mathfrak{j}$-nodes associated to a letter $R$ in $X_1$,  the resulting $l$-partition is in $\Uglov{e}{\bf s}$.

\end{Prop}
\begin{proof}
This directly follows from the definition of FLOTW $l$-partitions.

\end{proof}

\begin{Prop}\label{exif}
Let  ${\bf s}=(s_1,\ldots,s_l)\in \mathcal{S}_e^l$ and assume that $\ulambda\in\Uglov{e}{\bf s}$. Then there exists $\mathfrak{j}\in \mathbb{Z}/e\mathbb{Z}$ and  a removable $\mathfrak{j}$-node of $\ulambda$ such that all the  elements of $\mathcal{E}_{\mathfrak{j}} (\ulambda)$ greater than it are virtual nodes of the horizontal boundary. 

\end{Prop}
\begin{proof}
 Let $\ulambda\in\Uglov{e}{\bf s}$. By \cite[Lemma 5.7.13]{GJ},  there exist a removable $\mathfrak{j}$-node $\gamma=(a,b,c)$ such that the part $\lambda_a^c$ is the greater part of $\ulambda$ and such that if $\gamma'=(a',b',c')$  is a $\mathfrak{j}-1$-node in the vertical boundary of $\ulambda$ then $\lambda_a^c>\lambda_{a'}^{c'}$. We can also assume that $\gamma$ is a maximal $\mathfrak{j}$-node  toward the nodes with this property. 

By hypothesis, there cannot exist an addable $\mathfrak{j}$-node $\eta=(a'',b'',c'')$ such that $\eta>\gamma$ otherwise  
 we would have $\lambda_{a''}^{b''}=\lambda_a^c+1$  by Lemma \ref{L1}, and thus $\eta=(a'',b''-1,c'')$ would be a $\mathfrak{j}-1$ node on a part equal to $\lambda_a^c$. Thus, if  $\gamma_1\in  \mathcal{E}_{\mathfrak{j}} (\ulambda)$ is such that $\gamma_1> \gamma$ then $\gamma$ must   be on the boundary of $\ulambda$.

  If it is on the vertical boundary without being removable then we have again $\lambda_a^c=\lambda_{a_1}^{c_1}$ and we have a node with residue $\mathfrak{j}-1$ on a part equal to $\lambda_a^c$.
 It  is not removable by the maximality of $\gamma$. Thus, it is on the horizontal boundary and as  the part 
  $\lambda^c_a$ is maximal, it must be virtual. 
 As a conclusion, there exists a node of nature $R$
 such that all the  extended or addable $\mathfrak{j}$-nodes greater than it are virtual nodes of the horizontal boundary.

\end{proof}

\subsection{Proof of the main theorem in the case of FLOTW multipartitions}

The main result in the case of FLOTW $l$-partitions follows now quite easily from the above study. Let us start with the direct implication. Note that we will only need Proposition \ref{convA} below in the sequel for the proof of our general theorem. However,   a proof of the theorem for FLOTW multipartitions gives a good hint of how   things will work in general. So we give a full proof of it.  

\begin{Prop}\label{princ}
Let  ${\bf s}=(s_1,\ldots,s_l)\in \mathcal{S}_e^l$ and assume that $\ulambda\in\Uglov{e}{\bf s}$ then it is a staggered $l$-partition with respect to $(e,{\bf s})$.

\end{Prop}

\begin{proof}
Let us assume that  $\ulambda\in \Uglov{e}{\bf s}$. Then by Proposition \ref{exif}, there  exists a staggered sequence of $\mathfrak{j}$-nodes.  
If the word   $v_{\mathfrak{j}} (\ulambda)$ has no node of type $A$ then it has  at least one node of type $R$ and if we remove all the removable $\mathfrak{j}$-node, Lemma \ref{princ0} ensures that we obtain a FLOTW $l$-partition.

Otherwise we have $v_1 (\ulambda)=X_1 A X_2$  with $X_1$ and $X_2$ two subwords in the letters $\{A,R,B_{ver},B_{hor}\}$  with at least one removable node in $X_2$. By Lemma \ref{l20} and Lemma \ref{l2},  
  no node in $X_1$ can be $(1)$ or $(2)$-connected with a node in $X_2$ 
 and we can thus conclude by Proposition \ref{princ0}.

 \end{proof}
  \begin{Rem}
  In fact, we can slightly generalize the above proposition. Assume that   ${\bf s}=(s_1,\ldots,s_l)\in \mathcal{S}_e^l$ and assume that $\ulambda\in\Uglov{e}{\bf s}$
   then there exists $\mathfrak{j}\in \mathbb{Z}/e\mathbb{Z}$ such that 
  $v_{\mathfrak{j}} (\ulambda)=X_1 X X_2$ with  $X_1$ and $X_2$ two subwords in the letters $\{A,R,B_{ver},B_{hor}\}$ 
 and $X\in \{B_{ver},A\}$ and such that $X_2$ admits at least one element of type $R$. The above results show that if we remove all the nodes associated to the removable nodes is $X_2$ then we still obtain an $l$-partition in   $\Uglov{e}{\bf s}$.

  \end{Rem}

Now let us study the reversed implication. The following result is a little bit stronger than what we need for the next proposition but it will be crucial in the following to have it in this ``strong'' form.

\begin{Prop}\label{convA}
Let  ${\bf s}=(s_1,\ldots,s_l)\in \mathcal{S}_e^l$ and assume that $\ulambda\in\Uglov{e}{\bf s}$.
 Let $\mathfrak{j}\in \mathbb{Z}/e\mathbb{Z}$. 
  Consider the word 
 $\overline{w}_{\mathfrak{j}}(\ulambda)$  (see \S \ref{ra}) 
 Let $\gamma_q<\gamma_{q-1}<\ldots <\gamma_1$ be the addable  $\mathfrak{j}$-nodes corresponding  to 
 the letters $A$ in this word. For all $r\in \{1,\ldots,q\}$, if we add $\gamma_r$, \ldots, $\gamma_1$ to the $l$-partition  
  $\ulambda$ then  the resulting $l$-partition $\umu$ is in  $\Uglov{e}{\bf s}$
\end{Prop}
\begin{proof}
Let  $\umu$ be the $l$-partition obtained by adding $\gamma_r,\ldots,\gamma_1$ to $\ulambda$. We need to check the two conditions of FLOTW $l$-partition
\begin{itemize}
\item Assume that $\lambda^{c}_{i}=\lambda^{c+1}_{i+s_{c+1}-s_c}$ for $c\in \{1,\ldots,l-1\}$ and $i\in \mathbb{N}_{>0}$, and that 
 we have $\gamma_i=(i+s_{c+1}-s_c,\lambda^{c+1}_{i+s_{c+1}-s_c}+1,c+1)$ then the node 
 $\eta=(i,\lambda^c_i+1,c)$ must be addable and we have $\eta>\gamma_i$. Note also that they are consecutive so that  we must have $\eta=\gamma_{i-1}$ and the condition is still satisfied. The case where $\lambda^{l}_{i}=\lambda^{1}_{i+s_{1}-s_l+e}$ is similar.

 \item Assume that the set of residues of the nodes in the vertical boundary of 
parts of length $k\in \mathbb{Z}_{>0}$ is $\{0,1,\ldots,e-1\}\setminus \{\mathfrak{j}\}$ and that 
 $\gamma_s=(a,k,c)$ is an addable node of residue $\mathfrak{j}$ for some $s\in \{1,\ldots,q\}$. Note that  there exists a node 
  $(a',k,c')$ on the vertical boundary of $\ulambda$ with  residue $\mathfrak{j}-1$.
 For all such node, we have  that 
  $\eta=(a',k+1,c')$ must be an addable $\mathfrak{j}$-node. By Lemma \ref{L1}, we moreover have $\gamma_s< \eta$. We want to show that $\eta$ is in the sequence $(\gamma_q,\ldots,\gamma_1)$. This will 
 show that    the set of residues of the nodes of $\umu$ in the vertical boundary of 
parts of length $k$ is $\{0,1,\ldots,e-1\}\setminus \{\mathfrak{j}-1\}$. 
   To do this, we need to show that we don't have any removable $\mathfrak{j}$-node between $\eta$ and $\gamma_s$. We will then be able to conclude 
   by the definition of  $\overline{w}_{\mathfrak{j}}(\ulambda)$.
   
   Assume that we have a removable $\mathfrak{j}$-node $\eta_1=(a_1,b_1,c_1)$ such that $\gamma_s<\eta_1<\eta$. 
    By Lemma \ref{L1} and by the hypothesis, we necessarily have $b_1=k+1$. It is clear that, adding, all
    possible  addable $\mathfrak{j}$-nodes of the form $(a'',k,c'')$ (with $a'\in \mathbb{Z}_{>0}$ and $c''\in \{1,\ldots,l\}$)  gives an $l$-partition satisfying $(1)$ in Definition \ref{flotw}, so we can use 
    Remark \ref{LLrem} to   deduce that $\textrm{cont}(\gamma_s)=\textrm{cont} (\eta)$ or  
   $\textrm{cont}(\gamma_s)=\textrm{cont} (\eta)-e$.

   Assume that  $\textrm{cont}(\gamma_s)=\textrm{cont} (\eta)$  then we have $c_1>{c'}$ from what we deduce $a_1=a'+s_{c_1}-s_{c'}$ because $\textrm{cont}(\eta_1)=\textrm{cont} (\eta)$ and thus $\lambda^{c'}_{a'}\geq \lambda^{c_1}_{a_1}$ which is a contradiction. Assume that   $\textrm{cont}(\gamma_s)=\textrm{cont} (\eta)-e$.
   Then we obtain $a'=a+s_{c'}-s_c-e$. If $c> c'$ we obtain   $\lambda^{c'}_{a'}\geq \lambda^{c}_{a-e}$ which is impossible. Thus we have $c\leq c'$ and thus $c_1<c'$. We now have $a_1=a'+s_{c_1}-s_{c'}+e$ and thus
    $\lambda^{c'}_{a'}\geq \lambda^{c_1}_{a_1}$ which is a contradiction.

\end{itemize}

\end{proof}

\begin{Cor}
Let  ${\bf s}=(s_1,\ldots,s_l)\in \mathcal{S}_e^l$ and assume that $\ulambda$ is a staggered $l$-partition then we have $\ulambda\in\Uglov{e}{\bf s}$.

\end{Cor}
\begin{proof}
This is done by induction on $n$. Of course, if $n=0$, the result is trivial as the empty multipartition is in $\Uglov{e}{\bf s}$. Assume that $\ulambda$ is a staggered $l$-partition and remove the removable nodes of an associated staggered sequence. By induction, we obtain $l$-partition in $\Uglov{e}{\bf s}$. Then by Proposition \ref{convA}, we obtain that $\ulambda$ is in 
$\Uglov{e}{\bf s}$. The result follows. 

\end{proof}

\section{Crystal isomorphisms}\label{isosec}

The aim  of this section is to understand the links between the various types of Uglov $l$-partitions. 
 We will need to recall and develop results from \cite{JL}. We refer to this paper for details. In this section, we fix $e\in \mathbb{Z}_{>0}$.

\subsection{Action of the extended affine symmetric group} 
\begin{abs}
We denote by $P_l:=\mathbb{Z}^l$ the $\mathbb{Z}$-module with  standard basis $\{y_i\ |\ i=1,\ldots, l\}$. For $i=1,\ldots,l-1$, we denote by $\sigma_i$ the transposition $(i,i+1)$ of  $\mathfrak{S}_l$. Then the extended affine symmetric group 
$\widehat{\mathfrak{S}}_l$ can be seen as 
  the semi-direct product $P_l \rtimes \mathfrak{S}_l$ where the relations  are given by $\sigma_i y_j=y_j \sigma_i$ for $j\neq i,i+1$ and $\sigma_i y_i \sigma_i=y_{i+1}$ for $i=1,\ldots,l-1$ and $j=1,\ldots,l$.  
This group acts faithfully on $\mathbb{Z}^l$  by setting for any ${{{\bf s}}}=(s_{1},\ldots ,s_{l})\in 
\mathbb{Z}^{l}$: %
$$\begin{array}{rcll}
\sigma _{c}.{{{\bf s}}}&=&(s_{1},\ldots ,s_{c-1},s_{c+1},s_{c},s_{c+2},\ldots ,s_{l})&\text{for }c=1,\ldots,l-1 \text{ and }\\
y_i.{{{\bf s}}}&=&(s_{1},s_{2},\ldots,s_i+e,\ldots ,s_{l})&\text{for }i=1,\ldots,l.
\end{array}$$
For any $c\in \{0,...,l-1\}$, we set $z_{c}=y_{1}\cdot \cdot \cdot y_{c}.$
Write also $\xi =\sigma _{l-1}\cdot \cdot \cdot \sigma _{1}$ and $\tau
=y_{l}\xi$  then  $\widehat{\mathfrak{S}}_{l}$ is generated by the transpositions $%
\sigma _{c}$ with $c\in \{1,...,l-1\}$ and $\tau$. We have:
$$\tau. (s_1,\ldots,s_l)=(s_2,\ldots,s_{l},s_1+e).$$
A fundamental domain for this action is contained in the set 
$$\mathcal{S}_{e}^l:=\left\{(s_1,\ldots  ,s_{l})\in \mathbb{Z}^l\ |\ 0\leq s_i-s_j\leq e\ 1\leq i\leq j\leq l  \right\},$$
that we have already met in the last section. 
\end{abs}
\begin{abs}\label{qua}

Let $\mathfrak{g}$ be the quantum affine algebra of type $A^{(1)}_{e-1}$. Let ${\bf s}=(s_1,\ldots,s_l)\in \mathbb{Z}^l$. Then one can set an action (depending on ${\bf s}$) of $\mathfrak{g}$ on the Fock space,   the $\mathbb{Q}$-vector space with basis the set of  all the $l$-partitions $\Pi^l$.  In particular, the action of the Chevalley operators $e_{\mathfrak{j}}$ and  $f_{\mathfrak{j}}$ (with $\mathfrak{j}\in \mathbb{Z}/e\mathbb{Z}$)    on the Fock space is given by:
 $$\forall \mathfrak{j}\in \mathbb{Z}/e\mathbb{Z},\forall \ulambda \in \Pi^l,\ f_\mathfrak{j} \ulambda=\sum_{[\umu]=[\ulambda]\sqcup\{\gamma\},\ \operatorname{res}(\gamma)=\mathfrak{j}} \umu,  e_\mathfrak{j}\ulambda=\sum_{[\ulambda]=[\umu]\sqcup\{\gamma\},\ \operatorname{res}(\gamma)=\mathfrak{j}} \umu.$$

 The submodule $V_{\bf s}$ generated by the empty $l$-partition is then an irreducible highest weight module with weight $\Lambda_{s_1}+\ldots +\Lambda_{s_l}$ (where $\Lambda_i$ for $0\leq i\leq e-1$ denote the fundamental weights). 

One can construct the crystal graph of this module using the above combinatorics. It turns out that the vertices of this crystal (which label the associated Kashiwara-Lusztig canonical basis) are given by the set of Uglov $l$-partitions associated with ${\bf s}$. 
\end{abs}

\begin{abs}\label{iso}  
 One can see that the isomorphism class of  $V_{\bf s}$  only depends on the action of ${\bf s}$ modulo the above  action of the extended affine symmetric group $\widehat{\mathfrak{S}}_l$ on $\mathbb{Z}^l$.
 As a consequence, we have a bijection between the sets of Uglov $l$-partitions associated to elements of $\mathbb{Z}^l$ in the same orbit:
$$\forall \sigma\in  \widehat{\mathfrak{S}}_l,\  
 \Psi^e_{{\bf s}\to  \sigma.{\bf s}}: \Psi^e_{{\bf s}} (n)\to \Psi^e_{\sigma_i. {\bf s}}(n).$$

This bijection corresponds to a crystal isomorphism. 
 It can be computed recursively on the rank of the $l$-partitions as follows. Let $\ulambda \in \Uglov{e}{\bf s}$ then there exists $\gamma$  a good removable $\mathfrak{j}$-node 
 for $\ulambda$. Let $\ulambda'\in \Uglov{e}{\bf s}$ be the $l$-partition obtained  by removing $\gamma$ from $\ulambda$. By induction, we know $\umu':= \Psi^e_{{\bf s}\to  \sigma.{\bf s}} (\ulambda')$. Then 
  the crystal graph theory ensures that there exists an addable  $\mathfrak{j}$-node $\gamma'$ 
   for $\umu'$ such that $\gamma'$ is a good removable node for  the $l$-partition $\umu$ obtained by adding $\gamma'$ to $\umu'$. We have  
 $\umu\in \Uglov{e}{\sigma.{\bf s}}$. We then have $\umu:= \Psi^e_{{\bf s}\to  \sigma.{\bf s}} (\ulambda)$.

Of course, this way of computing the bijection may be difficult to handle when the rank gets bigger. The work \cite{JL} shows how one can compute this crystal isomorphism more easily and non recursively.  
 
 By the above discussion, to understand these bijections, it suffices to understand the case where 
$\sigma= \sigma_i$, the transposition $(i,i+1)$ for $i=1,\ldots,l-1$ and the case where  $\sigma=\tau$.

\begin{abs}\label{red}
We claim  that to obtain  
 the whole class of Uglov 
 $l$-partitions $\Uglov{e}{\bf s}$ for all ${\bf s}\in \mathbb{Z}^l$, it suffices to 
\begin{enumerate}
\item describe the elements of  $\Uglov{e}{\bf s}$ where 
${\bf s}\in \mathcal{S}_{e}^l$,
\item for all $i\in \{1,\ldots,l\}$, use the isomorphisms $\Psi^e_{{\bf s}\to \sigma_i . {\bf s}}$ for ${\bf s}=(s_1,\ldots,s_l)\in \mathbb{Z}^l$, 
 satisfying  $s_i\leq s_{i+1}$,
\item  use the isomorphisms $\Psi^e_{{\bf s}\to \tau. {\bf s}}$ for all ${\bf s}=(s_1,\ldots,s_l)\in \mathbb{Z}^l$ such that $s_l\leq s_1+e$.  
\end{enumerate}
First, it follows from an easy induction argument that if ${\bf s}=(s_1,\ldots,s_l)\in \mathbb{Z}^l$ is such that 
$s_1\leq s_2\leq \ldots \leq s_l$ then one can obtain all the elements of the form
 $(s_{\sigma (1)},\ldots,s_{\sigma (l)})$ with $\sigma\in \mathfrak{S}_l$ by applying isomorphisms as in $(2)$. 
 
Assume that ${\bf s}'\in \mathbb{Z}^l$ and let ${\bf s}\in \mathcal{S}_{e}^l$  in the orbit of ${\bf s}'$ modulo $\widehat{\mathfrak{S}}_l$. 
 By hypothesis, $\Uglov{e}{\bf s}$ is known. By Remark \ref{translation}, one can assume that 
$${\bf s}'=(s_{\sigma (1)}+k_{\sigma (1)}e,\ldots,s_{\sigma (l)}+k_{\sigma (l)}e),$$
where $(k_1,\ldots,k_l)$ is a collection of positive integers and we can assume that
$$s_{\sigma (1)}+k_{\sigma (1)}e\leq \ldots \leq s_{\sigma (l)}+k_{\sigma (l)}e.$$
If $k_{\sigma (l)}=0$ then ${\bf s}'={\bf s}$ and we are done. Otherwise, 
note that we have:
$$\tau. (s_{\sigma (l)}+(k_{\sigma (l)}-1)e, s_{\sigma (1)}+k_{\sigma (1)}e,\ldots,s_{\sigma (l-1)}+k_{\sigma (l-1)}e)={\bf s}'.$$
with $s_{\sigma (l-1)}+k_{\sigma (l-1)}e\leq s_{\sigma (l)}+(k_{\sigma (l)}-1)e+e$. 
The result now follows  by induction on $\sum_{1\leq i\leq l} k_i$. 
\end{abs}

\subsection{Explicit computations}

Let ${\bf s}\in \mathbb{Z}^l$. We now want to show how we can explicitly compute the isomorphisms
$ \Psi^e_{{\bf s}\to  \tau.{\bf s}} $ and $ \Psi^e_{{\bf s}\to  \sigma_c.{\bf s}} $ (with $1\leq c\leq l-1$).
For the first map, this is easy:

\begin{Prop}[Jacon-Lecouvey \cite{JL}]\label{tau}
 For all $\ulambda=(\lambda^1,\ldots,\lambda^l)\in \Uglov{e}{\bf s}$, we have:
 $$ \Psi^e_{{\bf s}\to  \tau.{\bf s}} (\ulambda)
 =(\lambda^2,\ldots,\lambda^{l},\lambda^1).$$
 \end{Prop}

\end{abs}

\begin{abs}\label{algo}
In \cite{JL}, formulae for the bijections   $\Psi^e_{{\bf s}\to  \sigma_i .{\bf s}}$ have been obtained. We here 
 translate these formulae in terms of extended Young diagrams (see also \cite{J}). 
 Set  ${\bf s}':=\sigma_i.  {\bf s}=(s_1,\ldots, s_{i+1},s_i,\ldots s_i)$. Let $\ulambda\in \Uglov{e}{\bf s}$. We have:
 $$\Psi^e_{{\bf s}\to {\bf s}'}(\ulambda)=(\lambda^1,\ldots,\lambda^{i-1},\widetilde{\lambda}^{i+1},\widetilde{\lambda}^{i},\lambda^{i+2},\ldots,\lambda^l),$$
and we now explain how one can obtain $(\widetilde{\lambda}^{i+1},\widetilde{\lambda}^{i})$ 
 from $({\lambda}^{i},{\lambda}^{i+1})$.

First,  We write $\lambda^i=(\lambda^i_1,\ldots,\lambda^i_{m_i})$ with $m_i\in \mathbb{Z}_{>0}$ and 
  $\lambda^{i+1}=(\lambda^{i+1}_1,\ldots,\lambda^{i+1}_{m_{i+1}})$ with $m_{i+1}\in \mathbb{Z}_{>0}$. We then define
  $$m:=\text{Max} (m_{i+1}+s_i-s_{i+1},m_i).$$
  Now adding as many zero part as necessary we slightly abuse notation by assuming  that 
 $\lambda^i=(\lambda^i_1,\ldots,\lambda^1_{m})$  and 
  $\lambda^{i+1}=(\lambda^{i+1}_1,\ldots,\lambda^{i+1}_{m -s_i+s_{i+1}})$.

As we have already seen, it suffices to consider the case   $s_i\leq s_{i+1}$. We will now consider the nodes $(a,\lambda_{a}^{i},i)$ on the vertical boundary of $\lambda^{i}$  from $a=1$ to $a=m$ and construct the new partitions $\widetilde{\lambda}^i$ and 
 $\widetilde{\lambda}^{i+1}$ 
  step by step by moving some boxes of the Young diagram of $(\lambda^i,\lambda^{i+1})$:
 \begin{itemize}
 \item Let us consider the  node $\gamma_1=(1,\lambda^{i}_{1},i)$, 
  and the associated content $\operatorname{cont} (\gamma_1)$. We consider the node $\gamma_1'=(a_1,\lambda^{i+1}_{a_1},i+1)$ on the vertical boundary 
   of $\lambda^{i+1}$ such that the content of $\gamma_1'$ is maximal among the  
   contents of the nodes which are smaller than the content of $\gamma_1$. Then we move all the nodes of contents 
 $\operatorname{cont}(\gamma_1')+1$, \ldots ,$\operatorname{cont}(\gamma_1)$ from the part
  $\lambda^i_{1}$ to the part $\lambda^{i+1}_{a_1}$. 
  
\item    Let us consider the  node $\gamma_2=(2,\lambda^{i+1}_{2},i)$, 
  and the associated content $\operatorname{cont} (\gamma_2)$. We consider the node $\gamma_2'=(a_2,\lambda^{i+1}_{a_2},i+1)$ with $a_2\neq a_1$  on the vertical boundary 
   of $\lambda^{i+1}$ such that the content of $\gamma_2'$ is maximal among the  
   contents of the nodes which are smaller than $ \operatorname{cont} (\gamma_2)$. Then we move all the nodes of content 
 $\operatorname{cont}(\gamma_2)+1$, \ldots ,$\operatorname{cont}(\gamma_2')$ from the part
  $\lambda^i_{2}$ to $\lambda^{i+1}_{a_2}$.
   \item We continue this process until we reach 
   $\gamma_{m}=(m,\lambda^{i}_{m},i)$.
   \end{itemize}
   At the end of the process, the bipartition $(\lambda^i,\lambda^{i+1})$ becomes $(\widetilde{\lambda}^{i+1},\widetilde{\lambda}^{i})$ 

\end{abs}
\begin{exa}
Take $(s_1,s_2)=(0,1)$, $e=3$ and $\ulambda=(3.2,1)$ which is  in $\Uglov{e}{\bf s}$.   The associated extended  Young diagram is:
\vspace{1cm}  
  
  \centerline{\small{
 \Bigg(\  \ \ \ \ytableausetup
{mathmode}
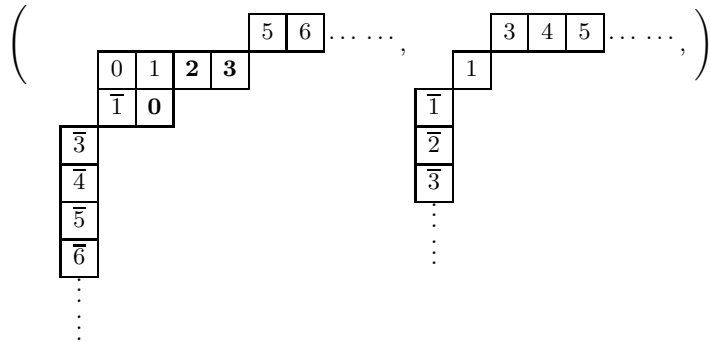
\begin{ytableau}
\none  &  \none  & \none  & \none  & \none & {5} & {6} & \none[\dots]& \none[\dots] \\
\none  &  { 0} & {1} & {\bf  2}& {\bf 3} \\
\none & \overline{ 1} & {\bf  0} \\
\overline{ 3}  \\  
\overline{4}  \\
\overline{5}  \\
\overline{6}  \\
\none[\vdots] \\
 \none[\vdots] \\
\end{ytableau},
\begin{ytableau}
\none  &  \none  & 3 & 4 & 5 & \none[\dots]& \none[\dots] \\
\none  &  { 1} \\  
\overline{1} \\
\overline{2}  \\
\overline{3}  \\
\none[\vdots] \\
 \none[\vdots] \\
\end{ytableau},
\Bigg)}}

We start with the node $(1,4,1)$. The node that we have to associate to it is $(2,1,2)$, and we have to move 
 the nodes $(1,3,1)$ and $(1,4,1)$ to the part $\lambda^2_1$. Then, for 
 $(1,4,1)$, we have to consider the node $(3,0,2)$ and move $(1,4,1)$ to $\lambda^2_2$. Then we are done:
 
 \vspace{1cm}  
  
  \centerline{\small{
 \Bigg(\  \ \ \ \ytableausetup
{mathmode}
\begin{ytableau}
\none  &  \none  & \none & \none & 5 & \none[\dots]& \none[\dots] \\
\none  &  { 1} & {\bf 2} & {\bf 3}\\  
\none & {\bf 0} \\
\overline{2}  \\
\overline{3}  \\
\none[\vdots] \\
 \none[\vdots] \\
\end{ytableau},
\begin{ytableau}
\none  &  \none  & \none  & 3  & 4 & {5} & {6} & \none[\dots]& \none[\dots] \\
\none  &  { 0} & {1}  \\
\none & \overline{ 1}   \\
\overline{ 3}  \\  
\overline{4}  \\
\overline{5}  \\
\overline{6}  \\
\none[\vdots] \\
 \none[\vdots] \\
\end{ytableau}
\Bigg)}}

So we obtain that $\Psi^3_{(0,1)\to (1,0)} (4.2,1)=(3.1,2.1)$. 
\end{exa}

\section{Combinatorial study}

In this section, we study combinatorially in details  the effect of the isomorphisms of \S \ref{iso} on the sets of Uglov multipartitions. 
 To do this,  recall the notations adopted in \S \ref{exte}, 
  in the following sections, we take ${\bf s}=(s_1,\ldots,s_l)\in \mathbb{Z}^l$  and  $\ulambda \in \Uglov{e}{\bf s}$.

\subsection{Nature of nodes and crystal isomorphisms }  
  
For $\mathfrak{j}\in \mathbb{Z}/e\mathbb{Z}$, the aim will be  to understand how the word  $v_{\mathfrak{j}}(\ulambda)$ 
is transformed after application 
of a crystal isomorphism.  Let ${\bf s}' \in \mathbb{Z}^l$ be in the orbit of ${\bf s}$ modulo the action of $\mathfrak{\widehat{S}}_l$ and denote $\umu:=\Psi^e_{{\bf s}\to  {\bf s}'} (\ulambda)$.

\begin{abs}\label{tau2}
We here study how the word  $v_{\mathfrak{j}} (\ulambda)$ is affected after application 
 of the crystal isomorphism $\Psi_{{\bf s} \to {\tau.{\bf s}}}^e$. This case is in fact easy. Indeed, by Proposition \ref{tau}, 
 we have a bijection:
  $$\begin{array}{ccc}
  \mathcal{Y}^{\textrm{ext}} (\ulambda)\sqcup \{\textrm{addable nodes of }\ulambda\} &\to& \mathcal{Y}^{\textrm{ext}} (\Psi_{{\bf s} \to {\tau.{\bf s}}}^e(\ulambda)))\sqcup \{\textrm{addable nodes of }\Psi_{{\bf s} \to {\tau.{\bf s}}}^e(\ulambda)\}  \\
  (a,b,c) & \mapsto & (a,b,c-1)
  \end{array}$$
 (where the integer of the associated component is understood modulo $l$) and this bijection preserves the order $>$ and the nature of the nodes. This implies that $v_{\mathfrak{j}} (\ulambda)=v_{\mathfrak{j}} (\Psi_{{\bf s} \to {\tau.{\bf s}}}^e(\ulambda))$. It is clear that if $\gamma$ is normal then  $\gamma=(a,b,c)$ in $\ulambda$ is   associated to 
  the node $\gamma=(a,b,c-1)$ in $\Psi_{{\bf s} \to {\tau.{\bf s}}}^e(\ulambda)$.

\end{abs}

\begin{abs}\label{sig} Assume that ${\bf s}=(s_1,\ldots,s_l)\in \mathbb{Z}^l$ and that there exists $c\in \{1,\ldots,l-1\}$ such that $s_c\leq s_{c+1}$. 
 We now want to understand how the word  $v_{\mathfrak{j}}(\Psi^e_{{\bf s}\to \sigma_c . {\bf s}}(\ulambda))$ can be obtained  from 
$v_{\mathfrak{j}}(\ulambda)$. Of course, only pairs of consecutive elements of $\mathcal{E}_{\mathfrak{j}}(\ulambda)$ lying in components $c$ and $c+1$  can be modified.  So 
let us assume that $j\in \mathbb{Z}$ and denote $\umu:= \Psi^e_{{\bf s}\to \sigma_c . {\bf s}}(\ulambda)$. We will  consider the elements of  $\mathcal{E}_{\mathfrak{j}}(\ulambda)$   with content $j$ in $\lambda^c$ and the element of  $\mathcal{E}_{\mathfrak{j}}(\ulambda)$  
 with content $j$ in $\lambda^{c+1}$ and  perform the algorithm described above. In this case, we say that the two associated nodes are {\it comparable}. Both extended or addable nodes  are of  nature $A$, $R$, $B_{hor}$, $B_{ver}$.
The nature of the two (extended) nodes of content $j$ in $\mathcal{E}_{\mathfrak{j}}(\ulambda)$ 
 can be seen in $v_{\mathfrak{j}}(\ulambda)$: it is given  by a subword
  $Z_1 Z_2$ with 
 $(Z_1,Z_2)\in \{A, R, B_{hor}, B_{ver}\}^2$.
 After application of the map   $\Psi^e_{{\bf s}\to \sigma_i . {\bf s}}$, the nature of the extended nodes of content $j$ in component $c$ and $c'$ will change in general.  
  In $v_{\mathfrak{j}}(\umu)$, the nature of the  two extended nodes  of content $j$ in $\mathcal{E}_{\mathfrak{j}}(\umu)$   are given by a certain subword  $Z_1' Z_2'$ with 
 $(Z_1',Z_2')\in \{A, R, B_{hor}, B_{ver}\}^2$.

 The following table shows how the nature of two consecutive addable or extended nodes can be  transformed by the map $\Psi^e_{{\bf s}\to \sigma_c . {\bf s}}$. One can see that, for one of the $16$ cases below, several possibilities for the  associated pair in 
 $\umu$ may occur.   
 \begin{center}
\begin{tabular}{|c|c||c|c|}
   \hline
   \multicolumn{2}{|c||}{{\bf Node on} $\ulambda$} & \multicolumn{2}{c|}{\bf Node on $\umu$\ } \\
      \hline
  $Z_1$ &  $Z_2$ &    $Z_1'$ & $Z_2'$ \\
   \hline 
   \hline 
   $R$ &$R$& $R$ &$R$ \\
   \hline
   $A$ &$R$& $A$ & $R$ \\
   \hline
   \multirow{2}*{$B_{ver}$} &\multirow{2}*{$R$}& $R$ & $B_{ver}$ ($\diamondsuit$) \\
    \cline{3-4}
    & &    $B_{ver}$& $R$  \\
   \hline
   $B_{hor}$ &$R$& $B_{hor}$ & $R$ \\
   \hline
   \multirow{2}*{$R$} & \multirow{2}*{$A$}& $R$ &$A$ \\
   \cline{3-4}
    & &    $B_{hor}$& $B_{ver}$ \\
   \hline
   $A$ &$A$& $A$ & $A$ \\
   \hline
   \multirow{2}*{$B_{ver}$} &\multirow{2}*{$A$}& $B_{ver}$ & $A$ \\
   \cline{3-4}
    & & $A$ & $B_{ver}$ \\
   \hline
   \end{tabular}\ \ \ \ \ 
\begin{tabular}{|c|c||c|c|}
   \hline
   \multicolumn{2}{|c||}{{\bf Node on} $\ulambda$} & \multicolumn{2}{c|}{\bf Node on $\umu$\ } \\
      \hline
  $Z_1$ &  $Z_2$ &    $Z_1'$ & $Z_2'$ \\
   \hline 
   \hline 
   $B_{hor}$ &$A$& $B_{hor}$ & $A$ \\
   \hline
    \multirow{2}*{$R$} & \multirow{2}*{$B_{hor}$}& $R$ &$B_{hor}$  \\
    \cline{3-4}
     & & $B_{hor}$ & $R$ ($\diamondsuit$)\\
   \hline
    \multirow{2}*{$A$} & \multirow{2}*{$B_{hor}$}& $A$ & $B_{hor}$ \\
    \cline{3-4}
     & & $B_{hor}$ & $A$ \\
   \hline
    \multirow{3}*{$B_{ver}$} & \multirow{3}*{$B_{hor}$}& $R$ & $A$ \\
    \cline{3-4}
    & & $B_{ver}$ & $B_{hor}$ \\
     \cline{3-4}
    & & $B_{hor}$ & $B_{ver}$  ($\diamondsuit$)\\
   \hline
   $B_{hor}$ &$B_{hor}$& $B_{hor}$ & $B_{hor}$ \\
   \hline
   $R$ &$B_{ver}$& $R$ &$B_{ver}$ \\
   \hline
   $A$ &$B_{ver}$& $A$ & $B_{ver}$ \\
   \hline
   $B_{ver}$ &$B_{ver}$& $B_{ver}$ & $B_{ver}$ \\
   \hline
   $B_{hor}$ &$B_{ver}$& $B_{hor}$ & $B_{ver}$ \\
   \hline
\end{tabular} 
\end{center}

\end{abs}

\vspace{0,5cm}

Below, We  study more precisely in which  configuration the situation marked above by a $(\diamondsuit)$ 
 occurs.

\begin{Def}\label{adef} Let ${\bf s}\in \mathbb{Z}^l$  and let $\ulambda$ be an $l$-partition. 
 Let $\mathfrak{j}\in \mathbb{Z}/e\mathbb{Z}$.  
 Assume that $\gamma_1$ and    $\gamma_2$ are two extended  $\mathfrak{j}$-nodes in $\mathcal{E}_{\mathfrak{j}}(\ulambda)$  on, respectively,  components $c_1$ and $c_2$ and  such that $\gamma_2<_{co} \gamma_1$. 
 Let $\gamma_1^1$, \ldots, $\gamma_1^{n_1}$ be the nodes of the vertical boundary of 
  $\lambda^{c_1}$ such that 
  $$\operatorname{cont} (\gamma_1)< 
 \operatorname{cont} (\gamma_1^{1})< \ldots < 
  \operatorname{cont} (\gamma_1^{n_1}),$$
 and   let $\gamma_2^1$, \ldots $\gamma_2^{n_2}$ be the nodes of the vertical boundary of 
  $\lambda^{c_2}$ such that 
  $$\operatorname{cont} (\gamma_2)<
 \operatorname{cont} (\gamma_2^{1})< \ldots < 
  \operatorname{cont} (\gamma_2^{n_2}).$$
 Then we say that $\gamma_1$ and $\gamma_2$  are {\it well-adapted} if we have $n_2\geq n_1$ and we are in one of the following situation. 
 \begin{itemize}
\item We have $c_1\in \{1,\ldots,l-1\}$ so that we have $c_2=c_1+1$ (and the two nodes  have the same content), we have $s_{c_1}\leq s_{c_2}$ 
  and for all $i=1,\ldots,n_1$, we have :
  $$\operatorname{cont} (\gamma_1^{i})\geq 
\operatorname{cont} (\gamma_2^i).$$
\item We have $c_1=l$ so that $c_2=1$ (and $\operatorname{cont} (\gamma_2)=\operatorname{cont} (\gamma_1)+e$), we have $s_l\leq s_1+e$ and  
or all $i=1,\ldots,n_1$, we have :
  $$\operatorname{cont} (\gamma_1^i)-e\geq 
\operatorname{cont} (\gamma_2^i).$$
\end{itemize}
\end{Def}

One can prove the two following lemmas  quite elementary using  the procedure described in  \S \ref{algo}.

\begin{lemma}\label{a2}
  Let ${\bf s}\in \mathbb{Z}^l$  be such that $s_c\leq s_{c+1}$ for $c\in \{1,\ldots,l-1\}$ and let $\ulambda\in \Uglov{e}{\bf s}$. 
  Assume that $\gamma_1$ and $\gamma_2$ are well-adapted extended  $\mathfrak{j}$-nodes  on components $c$ and $c+1$ with the same content $j$ so that  $\gamma_2<_{co} \gamma_1$    then 
  the natures of the  extended nodes of content $j$ on components $c$ and $c+1$ in 
$\Psi^e_{{\bf s}\to \sigma_c . {\bf s}}(\ulambda)$  are respectively 
 \begin{itemize}
 \item $B_{ver}$ and $B_{hor}$ if $\gamma_1$ is of nature $B_{hor}$ and $\gamma_2$ of nature $B_{ver}$. 
 \item $B_{ver}$ and $R$ if $\gamma_1$ is of nature $R$ and $\gamma_2$ of nature $B_{ver}$. 
 \item $R$ and $B_{hor}$ if $\gamma_1$ is of nature $B_{hor}$ and $\gamma_2$ of nature $R$. 
 
 \end{itemize}
\end{lemma}

\begin{lemma}\label{a3}
  Let ${\bf s}\in \mathbb{Z}^l$  be such that $s_c\leq s_{c+1}$ for $c\in \{1,\ldots,l-1\}$ and let $\ulambda\in \Uglov{e}{\bf s}$. 
  Assume that $\gamma_1$ and $\gamma_2$ are  extended  $\mathfrak{j}$-nodes  in component $c$ and $c+1$ with the same content $j$ so that  $\gamma_2<_{co} \gamma_1$. Assume that 
  $\gamma_1$ and $\gamma_2$ are not well-adapted and that $\gamma_1$ is of nature 
  $R$ (resp. $B_{hor}$) and $\gamma_2$ of nature $B_{ver}$ (resp. $R$). Then 
  the natures of the  extended nodes of content $j$ in component $c$ and $c+1$ in 
$\Psi^e_{{\bf s}\to \sigma_c . {\bf s}}(\ulambda)$  are respectively 
$R$ (resp. $B_{hor}$) and $B_{ver}$ (resp. $R$). 

\end{lemma}

\begin{exa}
Take $(s_1,s_2)=(0,1)$, $e=3$ and $\ulambda=(5.4,3.2.2)$ which is  in $\Uglov{e}{\bf s}$.   The associated extended  Young diagram is:
\vspace{1cm}  
  
  \centerline{\small{
 \Bigg(\  \ \ \ \ytableausetup
{mathmode}
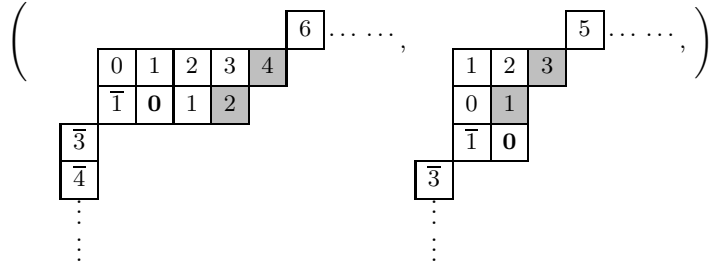
\begin{ytableau}
\none  &  \none  & \none  & \none  & \none & \none & {6} & \none[\dots]& \none[\dots] \\
\none  &  { 0} & {1} & {  2}& { 3} & *(lightgray)4  \\
\none & \overline{ 1} & {\bf   0} & 1 & *(lightgray) 2 \\
\overline{ 3}  \\  
\overline{4}  \\
\none[\vdots] \\
 \none[\vdots] \\
\end{ytableau},
\begin{ytableau}
\none  &  \none  & \none & \none & 5 & \none[\dots]& \none[\dots] \\
\none  &  { 1} & 2 & *(lightgray)3\\  
\none 	 & 0 & *(lightgray)1 \\
\none & \overline{1} & {\bf 0}   \\
\overline{3}  \\
\none[\vdots] \\
 \none[\vdots] \\
\end{ytableau},
\Bigg)}}

We take $\gamma_1=(2,2,1)$ which is of nature $B_{hor}$ and  
$\gamma_2=(3,2,2)$ which is of nature $R$. Observe that $\gamma_2 <_{co} \gamma_1$. We here have 
$n_1=n_2=2$ and $\gamma_1^1=(2,4,1)$ with content $2$, $\gamma_1^2=(1,5,1)$ with content $4$,
 $\gamma_2^1=(2,2,2)$ with content $1$, $\gamma_1^2=(1,3,2)$ with content $3$. So $\gamma_1$ and $\gamma_2$ are well-adapted.  Note that $\Psi^e_{{\bf s}\to \sigma_1 . {\bf s}}(\ulambda)$ is $(4.3.2, 4.3)$ with extended Young diagram:

 \vspace{1cm}  
  
  \centerline{\small{
 \Bigg(\  \ \ \ \ytableausetup
{mathmode}
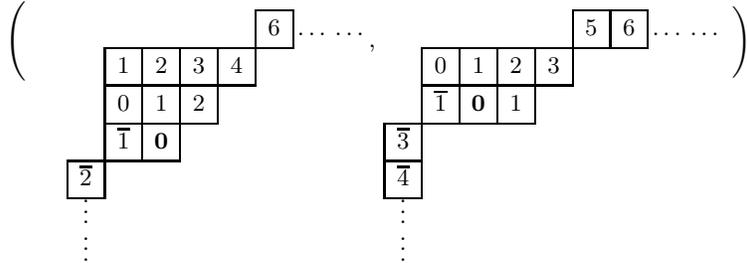
\begin{ytableau}
\none  &  \none  & \none & \none & \none & 6  & \none[\dots]& \none[\dots] \\
\none  &  { 1} & { 2} & { 3} & 4\\  
\none & {0} & 1 & 2\\
\none & \overline{1} & {\bf 0}    \\
\overline{2}  \\
\none[\vdots] \\
 \none[\vdots] \\
\end{ytableau},
\begin{ytableau}
\none  &  \none  & \none  & \none  & \none & {5} & {6} & \none[\dots]& \none[\dots] \\
\none  &  { 0} & {1}  & 2 & 3 \\
\none & \overline{ 1} & {\bf 0} & 1   \\
\overline{ 3}  \\  
\overline{4}  \\
\none[\vdots] \\
 \none[\vdots] \\
\end{ytableau}
\Bigg)}}

We see that the nodes of content $0$ in components $1$ and $2$ are of nature $R$ and $B_{hor}$ which is consistent with the above lemma. 
\end{exa}

\begin{lemma}\label{a0}   Let ${\bf s}\in \mathbb{Z}^l$  and let $\ulambda\in \Uglov{e}{\bf s}$.
 Assume that $\gamma_1=(a_1,b_1,c_1)$ and    $\gamma_2=(a_2,b_2,c_2)$  are well-adapted (extended) $\mathfrak{j}$-nodes in     $\mathcal{E}_{\mathfrak{j}}(\ulambda)$ with $\gamma_2<_{co} \gamma_1$. 
Assume that $\sigma=\tau$ or that $\sigma=\sigma_c$ with $c\in \{1,\ldots,l-1\}$ is such tat $s_c\leq s_{c+1}$. Set 
$\umu=\Psi^e_{{\bf s}\to \sigma .{\bf s}}(\ulambda)$. 
We denote by:
\begin{itemize}
\item $\gamma_1 '$ and  
 $\gamma_2'$ the (extended) nodes in $\mathcal{E}_{\mathfrak{j}}(\umu)$ with the same content as $\gamma_1$ and $\gamma_2$ if $\sigma=\sigma_c$. 
 \item $\gamma_1 '=(a_1,b_1,c_1-1)$ and  
 $\gamma_2'=(a_2,b_2,c_2-1)$ the (extended) nodes in $\mathcal{E}_{\mathfrak{j}}(\umu)$ if $\sigma=\tau$. 
\end{itemize}
 Then $\gamma_1'$ and $\gamma_2'$ are well-adapted extended $\mathfrak{j}$-nodes. 

\end{lemma}
\begin{proof}
Once again, we apply the procedure described in \S \ref{algo} if $\sigma=\sigma_c$ with $c\in \{1,\ldots,l-1\}$. The result is trivial by Proposition \ref{tau} if $\sigma=\tau$. 
\end{proof}

\subsection{Induction and crystal isomorphisms}

\begin{abs}\label{corres}
Let ${\bf s}\in \mathbb{Z}^l$ and 
${\bf s}'\in \mathbb{Z}^l$ be in the same orbit modulo the action of $\widehat{\mathfrak{S}}_l$. We take 
$\ulambda\in \Uglov{e}{\bf s}$ and $\umu:=\Psi^e_{{\bf s}\to {\bf s}'}(\ulambda)$.  
By \cite[\S 4]{JU2}, there is a one to one correspondence between the normal $\mathfrak{j}$-nodes
 of $\ulambda$ and the normal  $\mathfrak{j}$-nodes of $\umu$ (it suffices to consider the case $l=2$ which is treated in {\it op. cit.}, but this is also easily seen using the above algorithm). More precisely,  if we  have 
 $$\omega_{\mathfrak{j}} (\ulambda)=A\ldots A \underbrace{R\ldots R}_{p\text{ times}}$$
 for $p\in \mathbb{Z}_{>0}$. Then 
 $$\omega_{\mathfrak{j}} (\umu)=A\ldots A \underbrace{R\ldots R}_{p\text{ times}}$$
 If we denote by $\gamma_1$, \ldots $\gamma_p$ the normal nodes of $\ulambda$ (written in increasing order) corresponding to the sequence of $R$ above,  
 then  one can canonically associate  to each  $\gamma_j$ 
 a removable normal node of $\umu$  which is denoted by $\gamma_j '$ such that $\gamma_1'$, \ldots, $\gamma_p'$ are the normal removable nodes of $\umu$ written in increasing order. In the following, we say that 
  $\gamma_j$ is {\it associated} to $\gamma_j'$. 

\end{abs}
\begin{lemma}\label{prec}
Under the above hypotheses, let   $s\in \{1,\ldots,p\}$ be such that $s=1$ or such that the node $\gamma_{s-1}$ is not consecutive to the node $\gamma_s$. Denote by 
 $\ulambda'$ (resp. $\umu'$) the $l$-partition obtained by removing the normal nodes associated to $\gamma_s$, \ldots,  $\gamma_p$ (resp. of $\gamma_s'$, \ldots,  $\gamma_p'$). Assume that $\ulambda'\in \Uglov{e}{\bf s}$, then  for all $c\in \{1,\ldots,l-1\}$ such that $s_c\leq s_{c+1}$, 
 we have 
  $\umu'\in \Uglov{e}{\sigma_c.{\bf s}}$ and $\umu' =\Psi^e_{{\bf s}\to \sigma_c . {\bf s}} (\ulambda')$.
\end{lemma}
\begin{proof}
We have to perform the  algorithm of \S \ref{algo}. We keep the 
 notations of the lemma. We have to check what happen on the components $c$ and $c+1$ 
  of  $\ulambda$ and $\umu$ and then check that  $\umu' =\Psi^e_{{\bf s}\to \sigma_c . {\bf s}} (\ulambda')$.
   This will imply that $\umu'\in \Uglov{e}{\sigma_c. {\bf s}}$ as $\ulambda'\in \Uglov{e}{\bf s}$.

  First note that two consecutive nodes on components $c$ and $c+1$ have always the same content $j$.
   So we can just focus on nodes of such content.  
\begin{itemize}
\item if there is only one removable normal node with content $j$, then the result is clear,
\item Assume that we have one removable node of content $j$ in each of the components. Performing the  
 algorithm, we see that the only problem may happen when one of the removable node   associated to $\{\gamma_s,\ldots,\gamma_p\}$
 is on the component $c$ and the removable node in the component $c+1$
 is not in $\{\gamma_s,\ldots,\gamma_p\}$. However, this case is impossible because two such nodes are consecutive and if 
 the node in the component $c$ is normal then so is the node in component $c+1$.  

\end{itemize}

\end{proof}

\begin{lemma}\label{prec2}
Under the above hypotheses, let   $t\in \{1,\ldots,p\}$ be such that $t=1$ or such that the node $\gamma_{t-1}$ and $\gamma_t$ are consecutive nodes with the same content but which are not well-adapted. Denote by 
 $\ulambda'$ (resp. $\umu'$) the $l$-partition obtained by removing the normal nodes associated to $\gamma_t$, \ldots,  $\gamma_p$ (resp. of $\gamma_t'$, \ldots,  $\gamma_p'$). Assume that $\ulambda'\in \Uglov{e}{\bf s}$, then  for all $c\in \{1,\ldots,l-1\}$ such that $s_c\leq s_{c+1}$,  
 we have 
  $\umu'\in \Uglov{e}{\sigma_c.{\bf s}}$ and $\umu' =\Psi^e_{{\bf s}\to \sigma_c . {\bf s}} (\ulambda')$.
\end{lemma}
\begin{proof}
We argue as in the proof of lemma \ref{prec}. However, we have to see what happen now if we $\gamma_t$ and $\gamma_{t-1}$ are in component $c$ and $c+1$ of the $l$-partition. Performing the algorithm of \S \ref{algo}, we see that the fact that the two nodes are not well-adapted implies the result. 

\end{proof}

\begin{lemma}\label{wa}
Let ${\bf s}\in \mathcal{S}_e^l$ and let $\ulambda\in \Uglov{e}{\bf s}$. Let $\gamma_1 <\gamma_2$ be two consecutive normal removable $\mathfrak{j}$-nodes with the same content. Assume that they are well-adapted. Let ${\bf s'}\in \mathbb{Z}^l$ be in the orbit of ${\bf s}$ modulo the action of $\widehat{\mathfrak{S}}_l$ and consider the normal removable  nodes associated to $\gamma_1$  and $\gamma_2$ in $\umu:=\Psi^e_{{\bf s}\to  {\bf s}'} (\ulambda)$. Then there cannot be any addable node between these two nodes.

\end{lemma}
\begin{proof}
Write ${\bf s}'=\sigma.{\bf s}$ with $\sigma\in \widehat{\mathfrak{S}}_l$.
We write $\sigma$ as a product of generators $\sigma_c$ with $c\in \{1,\ldots,l-1\}$ and $\tau$ satisfying the conditions  in \S \ref{red}. Then 
$\Psi^e_{{\bf s}\to  {\bf s}'}$ is a composition of isomorphisms 
$\Psi^e_{{\bf v}\to  \sigma_c. {\bf v}}$ (with then $v_c\leq v_{c+1}$)  and $\Psi^e_{{\bf v}\to  \tau. {\bf v}}$ 
  (with then $v_l\leq v_1+e$). 

Assume that we have an addable node between the normal removable  nodes associated to $\gamma_1$  and $\gamma_2$ in $\umu:=\Psi^e_{{\bf s}\to  {\bf s}'} (\ulambda)$.  By the table in \S \ref{sig}, we see that this is possible only if two consecutive nodes between nodes associated to $\gamma_1$ and $\gamma_2$ of nature $B_{ver}$ and $B_{hor}$ are transformed into nodes of nature $R$ and $A$ after application of a crystal isomorphism as above. Now combining Lemma \ref{a2}, Lemma \ref{a0} and Lemma \ref{a3}, we see that the nodes between nodes associated to $\gamma_1$ and $\gamma_2$ must be always well-adapted. By Lemma \ref{a2} again, well-adapted nodes of nature  $B_{ver}$ and $B_{hor}$ cannot be  transformed into nodes of nature $R$ and $A$. So this is a contradiction.

\end{proof}

\section{Proof of the main result}

We are now ready to give the proof of our main Theorem (which is Theorem \ref{main}). We begin with one implication of the main theorem. 

\subsection{Direct implication}

\begin{Prop}\label{direct}  Assume that $e\in \mathbb{Z}_{>1}$ and ${\bf s}=(s_1,\ldots,s_l)\in \mathbb{Z}^l$.
Assume that $\ulambda$ is a staggered $l$-partition with respect to $(e,{\bf s})$ then 
 $\ulambda$ is an Uglov $l$-partition with respect to $(e,{\bf s})$.

\end{Prop}
\begin{proof}
We argue by induction on the rank of $\ulambda$. If this rank is $0$, the result is of course true because $\uemptyset$ is an Uglov $l$-partition. 
 Assume thus that $\ulambda$ is a non empty staggered $l$-partition, then
  $\ulambda$ admits a 
  staggered sequence of nodes. Let  
 $(\gamma_1,\ldots,\gamma_r)$ be the removable nodes of this sequence (written in increasing order). We denote by $\mathfrak{j}\in \mathbb{Z}/e\mathbb{Z}$ the common residue of these nodes. Let $\ulambda'$ be the $l$-partition of $n-r$ obtained by removing the nodes 
 $(\gamma_1,\ldots,\gamma_r)$ from $\ulambda$. By induction, this is an Uglov $l$-partition.

Let us first assume that $v_{\mathfrak{j}} (\ulambda)$ has no nodes of nature $A$. It implies that 
 $\gamma_1$ is a good $\mathfrak{j}$-node for $\ulambda$, then $\gamma_2$ is a good removable 
   $\mathfrak{j}$-node for the $l$-partition obtained by removing $\gamma_1$ from  $\ulambda$
 etc ... we conclude that we can remove a sequence of good nodes to obtain an Uglov $l$-partition. This implies that $\ulambda$ is an Uglov $l$-partition.

So let us now assume that  $v_{\mathfrak{j}} (\ulambda)=X_1 A X_2$ where $X_{2}$ is a word in $R$, $B_{ver}$ and $B_{hor}$. 
 In addition we assume there are exactly $r$ letters $R$ in $X_2$ (with thus $r>0$). So we can write 
 $v_{\mathfrak{j}} (\ulambda')=X_1' A X_2'$ where $X_2'$ has exactly $r$ letters $A$ and no letters $R$. 
 
Let $\sigma \in \widehat{\mathfrak{S}}_l$ be such that  ${\bf s}':=\sigma.{\bf s}\in \mathcal{S}^l_e$. 
 Let $\umu'\in\Uglov{e}{{\bf s}'}$  be such that   $\Psi^e_{{\bf s}'\to {\bf s}}   (\umu')=\ulambda' $.
  We write $\sigma$ as a product of the generators $\tau$ and $\sigma_c$ with $c\in \{1,\ldots,l-1\}$ as in \S \ref{red}. Then 
   $\Psi^e_{{\bf s}'\to {\bf s}} $ is a composition of bijections of type
 $\Psi^e_{{\bf v}\to \sigma.{\bf v}} $     where ${\bf v}=(v_1,\ldots,v_l)$ is in the orbit of ${\bf s}$ and 
  $\sigma=\tau$ (with then $v_l\leq v_1+e$) or $\sigma=\sigma_c$ for $c\in \{1,\ldots,l-1\}$ (with then $v_c\leq v_{c+1}$) and we get a sequence of associated Uglov $l$-partitions $\unu'$ by applying successively these maps. 
  
  Regarding the table in \S \ref{sig}, we see that  $v_{\mathfrak{j}} (\unu')$ can be written as $X_1''X_2''$  where $X_1''$ and   $X_2''$ are words in the four letters such that if we delete all the occurrences  
   of $B_{ver}$ and $B_{hor}$ in $X_2''$ and then all the occurrences $RA$ in it, we obtain exactly $r$ 
    letters $A$ (and no letter $R$) corresponding to certain nodes $\gamma_1[\unu']$, \ldots, $\gamma_r[\unu']$. We will now distinguish two cases:

  \begin{itemize}
  \item Assume that for all these Uglov $l$-partitions $\unu'$, we can take 
   $X_1''$ so that the last letter of $X_1''$ is $A$.  
  This is in particular the case for $\unu'=\umu'$. So by Proposition \ref{convA},  the  $l$-partition $\umu$   obtained by adding the $\mathfrak{j}$-nodes $\gamma_1[\umu']$, \ldots, $\gamma_r[\umu']$
 to $\umu'$ is  in $\Uglov{e}{{\bf s}'}$. 
 Assume that $\Psi^e_{{\bf s}'\to \sigma.{\bf s}'}   (\umu')=\unu' $ for $\sigma \in \{\tau,\sigma_c,\ c=1,\ldots,l-1\}$ and consider $\unu$  the $l$-partition obtained by adding the $\mathfrak{j}$-nodes $\gamma_1[\unu']$, \ldots, $\gamma_r[\unu']$. If $\sigma=\tau$, it is clear that $\unu$ is Uglov. 
 Otherwise, consider the Uglov $l$-partition  $\Psi^e_{{\bf s}'\to \sigma.{\bf s}'}   (\umu)$. If we delete the $r$ greatest removable nodes from it, we get $\unu'$ by Proposition  \ref{prec} and this implies that $\unu=
 \Psi^e_{{\bf s}'\to \sigma.{\bf s}'}   (\umu)$. So this is an Uglov $l$-partition. By induction, we obtain that 
 $\Psi^e_{{\bf s}'\to {\bf s}}   (\umu)=\ulambda $
  which shows that $\ulambda$ is Uglov.

 \item Otherwise, this means that to get from the FLOTW $l$-partition $\umu'$ to $\ulambda'$, two consecutive nodes of  natures $B_{ver}$ and $B_{hor}$ has been transformed into associated nodes  of nature, respectively,  $R$ and $A$ (this is the only way to ``create'' an addable node after a crystal isomorphism regarding \S \ref{sig}). We begin by  arguing as above : 
 we can now write  $\widetilde{w}_{\mathfrak{j}} (\umu')=\widetilde{X}_1' \widetilde{X}_2'$ 
  as above. 
 By Proposition \ref{convA}, if we consider the $l$-partition obtained by adding the addable $\mathfrak{j}$-nodes $\gamma_1[\umu']$, \ldots, $\gamma_r[\umu']$ to $\umu'$, we obtain an $l$-partition $\umu$ in $\Uglov{e}{{\bf s}'}$. The problem is now to show that $\Psi^e_{{\bf s}'\to {\bf s}}   (\umu)=\ulambda $. 
      To do this, we argue as in the above case except that we use alternatively Lemma \ref{prec} and Lemma \ref{prec2} : 
for any multipartition $\unu'$ as above, if the $r$th normal removable greatest $\mathfrak{j}$-node is consecutive to the  $r+1$th removable greatest $\mathfrak{j}$-node (that is, in the case where we cannot use Proposition \ref{prec}), then both nodes are not well-adapted by Proposition \ref{wa}  and we can use proposition \ref{prec2}. We thus obtain  that $\Psi^e_{{\bf s}'\to {\bf s}}   (\umu)=\ulambda $
  which shows that $\ulambda$ is Uglov. 
  
  \end{itemize}

\end{proof}

\subsection{Reversed implication}

\begin{lemma}\label{pt}
Let  ${\bf s}=(s_1,\ldots,s_l)\in \mathbb{Z}^l$ and assume that $\ulambda\in\Uglov{e}{\bf s}$.
 Assume that for $\mathfrak{j}\in \mathbb{Z}/e\mathbb{Z}$, $\ulambda$ admits a 
  staggered sequence of $\mathfrak{j}$-nodes.  Assume that the sequence has $r$ removable nodes. 
Let  ${\bf s}'\in \mathcal{S}_e^l$ 
 be in the orbit of ${\bf s}$ modulo the action of $\widehat{\mathfrak{S}}_l$. We set 
  $\umu:=\Psi^e_{{\bf s}\to  {\bf s}'}(\ulambda)\in\Uglov{e}{{\bf s}'}$. Then 
we can remove at least $r$ removable $\mathfrak{j}$-nodes from $\umu$ and if we remove the $r$ greatest ones, the resulting $l$-partition is in $\Uglov{e}{{\bf s}'}$.

\end{lemma}
\begin{proof}
We can write 
$v_{\mathfrak{j}} (\ulambda)=X_1AX_2$ where $X_1$ and $X_2$ are two words in 
 $\{A,R,B_{ver},B_{hor}\}$ and $X_2$ has exactly $r$ letters $R$ and no letter $A$. 
 By \S \ref{corres}, as $\ulambda$ has at least $r$ normal removable nodes, this is also the case for $\umu$. 
 If $\umu$ has no other removable  nodes then the result is clear as the removable normal nodes can be successively seen as good nodes. Otherwise, 
  Let $\gamma$ be the $r$th greatest removable normal $\mathfrak{j}$-nodes of $\umu$ and let $\gamma'$ be the $r+1$th removable $\mathfrak{j}$-node.  We will show that $\gamma$ and $\gamma'$ are not $(1)$ or $(2)$ connected so that we can conclude by Lemma \ref{princ0}
  \begin{itemize}
  \item Assume that the nodes are $(1)$-connected. By the definition together with  Proposition \ref{period},
This implies that we are in one of the following two cases : 
\begin{itemize}
\item we have a set of nodes
$(\gamma_1,\ldots,\gamma_e)$ in the vertical boundary of the extended Young diagram of $\ulambda$, with 
$\gamma_k=\gamma$ and $\gamma_{k+1}=\gamma'$
 such that:
 for all $i\in \{1,\ldots,e-1\}\setminus\{k,k+1\}$, we have $\operatorname{cont}(\gamma_{i+1})=\operatorname{cont} (\gamma_i)+1$ and $c_{i+1}\geq c_i$.
Now it is easy to se that, applying any isomorphism  $\Psi^e_{{\bf s}'\to  {\bf s}}$, gives the existence of nodes  $(\widetilde{\gamma}_1,\ldots,\widetilde{\gamma}_e)$ with the same property
  and such that $\widetilde{\gamma}_k$ and $\widetilde{\gamma}_{k+1}$
 are the normal removable nodes associated to $\gamma$ and $\gamma'$. Now if we assume that there exists an addable node between  $\widetilde{\gamma}_k$ and $\widetilde{\gamma}_{k+1}$, this shows the existence of a period for the associated multipartition as defined in Proposition \ref{period} and the next remark. However, it follows from \cite[Prop. 5.1]{JL} that  an Uglov $l$-partition cannot have such a period. 
\item we have a set of nodes
$(\gamma_1,\ldots,\gamma_e)$ in the vertical boundary of the extended Young diagram of $\ulambda$, with 
$\gamma_1=\gamma$ and $\gamma_{e}=\gamma'$
 such that:
 for all $i\in \{1,\ldots,e-2\}$, we have $\operatorname{cont}(\gamma_{i+1})=\operatorname{cont} (\gamma_i)+1$ and $c_{i+1}\geq c_i$. In this case, we have  $\operatorname{cont}(\gamma_{1})=\operatorname{cont} (\gamma_e)-e$
 and $c_1<c_e$. Now, regarding the table in \S \ref{sig},  there must exist ${\bf s}''\in \mathbb{Z}^l$ in the orbit of ${\bf s}$ module $\widehat{\mathfrak{S}}_l$ and  a multipartition $\unu$ such that 
  $\Psi^e_{{\bf s}\to  {\bf s}''} (\ulambda)=\unu$ and such that $\unu$ satisfies the folloowing property. 
   $\unu$ admits a sequence 
  $(\widetilde{\gamma}_1,\ldots,\widetilde{\gamma}_e)$ in the vertical boundary of the extended Young diagram of $\ulambda$, such that  
 for all $i\in \{1,\ldots,e-2\}$, we have $\operatorname{cont}(\widetilde{\gamma}_{i+1})=\operatorname{cont} (\widetilde{\gamma}_i)+1$ and $c_{i+1}\geq c_i$ and such that $\widetilde{\gamma}_1$ is of nature $B_{ver}$. 
  But thus implies that  $\unu$ admits a period as defined in Proposition \ref{period}, this is a contradiction because $\unu$ must be in $\Uglov{e}{{\bf s}''}$. 
\end{itemize}

  \item Assume that the nodes are $(2)$-connected. If these two nodes are comparable (that is with the same content), we use Lemma \ref{wa}. This shows that there cannot be an addable node between the $r$th removable nodes of $\umu$ and the $r+1$th one. This is a contradiction. If they are not comparable, then we apply 
  $\Psi^e_{{\bf s}'\to  \tau .{\bf s}'}$ which makes them comparable and we argue as above.

  \end{itemize}

\end{proof}

The following lemma shows the existence of a staggered sequence of nodes for an Uglov $l$-partition.

\begin{lemma}\label{fi}
Let ${\bf s}\in \mathbb{Z}^l$ and let $\ulambda\in \Uglov{e}{\bf s}$ then there exists a removable $\mathfrak{j}$-node $\gamma$ such that there is no addable $\mathfrak{j}$-node $\eta$ such that $\eta>\gamma$. 
\end{lemma}
\begin{proof}
 Let $\ulambda\in\Uglov{e}{\bf s}$, we have a multicharge ${\bf s}'\in \mathcal{S}_e^l$ such that ${\bf s}$ and ${\bf s}'$ are in the same orbit modulo the action of the affine extended symmetric group. One can go from the  multipartition $\umu\in\Uglov{e}{{\bf s}'}$ 
 to 
 the multipartition $\ulambda$  using a finite sequence of isomorphims of the type \S \ref{iso}:
 $\Psi^e_{{\bf v}\to \sigma_c . {\bf v}}$ for ${\bf v}=(v_1,\ldots,v_l)$ in the orbit of ${\bf s}$ such that $v_c\leq v_{c+1}$  and $\Psi^e_{{\bf v}\to \tau . {\bf v}}$ for ${\bf v}$ in the orbit of ${\bf s}$.
 
As we have already seen in \S \ref{tau}, the isomorphism $\Psi^e_{{\bf v}\to \tau . {\bf v}}$ does not affect the words   $v_{\mathfrak{j}} (\unu)$ for $\unu\in\Uglov{e}{{\bf v}}$ .  
As $\umu\in\Uglov{e}{{\bf s}'}$, by the above discussion, there exist  $\mathfrak{j}\in \mathbb{Z}/e\mathbb{Z}$ and a $\mathfrak{j}$-node $\gamma$ as above such that 
 all the  extended or addable $\mathfrak{j}$-nodes greater than  $\gamma$ are in fact virtual nodes of the horizontal boundary. In other 
  words, the word  $v_{\mathfrak{j}} (\umu)$ is of the form $ZRY$ where $Z$ is a word in the four letters and $Y$ is 
   the infinite word in $B_{hor}$. Now, it is easy to see that a consecutive pair of nodes 
    of nature $XB_{hor}$ with $B_{hor}$ virtual is always transformed into a pair of nature 
     $XB_{hor}$ after application of the isomorphism of type  $\Psi^e_{{\bf v}\to \sigma_c . {\bf v}}$ if ${\bf v}=(v_1,\ldots,v_l)$  is such that $v_c\leq v_{c+1}$. 
    As a conclusion, using the table in \S \ref{sig},  we see that the word $v_{\mathfrak{j}} (\ulambda)$ is of the form
     $Z'RY'$ where $Z'$ is a word in the four letters and $Y'$ a word in $B_{hor}$ and $B_{ver}$. This concludes the proof.

\end{proof}

\begin{Prop}\label{mainu}
Let  $e\in \mathbb{Z}_{>1}$ and ${\bf s} \in \mathbb{Z}^l$. If $\ulambda$ is an Uglov $l$-partition with respect to $(e,{\bf s})$, then this is a staggered $l$-partition with respect to $(e,{\bf s})$. 
\end{Prop}
\begin{proof}
 By Lemma \ref{fi}, there exists a removable $\mathfrak{j}$-node $\gamma$ such that there is no addable $\mathfrak{j}$-node $\eta$ such that $\eta>\gamma$. This shows the existence of a staggered sequence of nodes. 
If we remove the removable nodes $\gamma_1<\ldots< \gamma_r$ of the associated   staggered sequence, we obtain an $l$-partition $\ulambda'$. 

Assume that we have no addable $\mathfrak{j}$-node for $\ulambda$. This means that $\gamma_r$ is a good removable node. If we remove it from $\ulambda$, $\gamma_{r-1}$ becomes a good removable node for the resulting $l$-partition which is thus an Uglov $l$-partition. Continuing in this way, we see that  $\ulambda'$ is  an Uglov $l$-partition with respect to $(e,{\bf s})$.

Assume that there exists an addable  $\mathfrak{j}$-node.  
Let  ${\bf s}'\in \mathcal{S}_e^l$ 
 be in the orbit of ${\bf s}$ modulo the action of $\widehat{\mathfrak{S}}_l$. We set 
  $\umu:=\Psi^e_{{\bf s}\to  {\bf s}'}(\ulambda)\in\Uglov{e}{{\bf s}'}$. By Lemma  \ref{pt}
we can remove at least $r$ removable $\mathfrak{j}$-nodes from $\umu$ and if we remove the $r$ greatest ones, the resulting $l$-partition  $\umu'$ is in $\Uglov{e}{{\bf s}'}$. 

We now show that $\Psi^e_{{\bf s}'\to  {\bf s}}(\umu')=\ulambda'$. This will show that $\ulambda'$ is Uglov and we can then conclude by induction. We can write  ${\bf s}=\sigma. {\bf s}'$  where $\sigma\in \widehat{\mathfrak{S}}_l$ is a product of the generators $\tau$ and $\sigma_c$ with $c\in \{1,\ldots,l-1\}$ as in \S \ref{red}. Then 
   $\Psi^e_{{\bf s}'\to {\bf s}} $ is a composition of bijections of type
 $\Psi^e_{{\bf v}\to \sigma.{\bf v}} $     where ${\bf v}=(v_1,\ldots,v_l)$ is in the orbit of ${\bf s}$ and 
  $\sigma=\tau$ (with then $v_l\leq v_1+e$) or $\sigma=\sigma_c$ for $c\in \{1,\ldots,l-1\}$ (with then $v_c\leq v_{c+1}$) and we get a sequence of associated Uglov $l$-partitions $\unu'$ by applying successively these maps. Let 
 $\gamma_1[\unu']<, \ldots, <\gamma_r[\unu']$ be the greatest removable normal $\mathfrak{j}$-nodes of $\unu'$. 
  Then if there exists a normal removable $\mathfrak{j}$-node $\eta$ such that $\eta<_{co}\gamma_1[\unu']$ then these two nodes cannot be well-adapted by Lemma \ref{wa}.  
 We conclude using Lemma \ref{prec2}.

\end{proof}

\section{Two consequences}\label{fkl}

The purpose of this section is to give two consequences of our main Theorem. This concerns mainly the sets of Kleshchev $l$-partitions.  We refer to \cite[Ch.6]{GJ} for details on the representations of quantum groups in affine type $A$ which are needed.

\subsection{Proof of the generalized Dipper-James-Murphy's conjecture}

\begin{abs}
As we have already mentioned, the  Uglov $l$-partitions 
have an interpretation in the context of quantum groups. Recall the notation of \S \ref{qua}. 
Let us  assume that  ${\bf s}=(s_1,\ldots,s_l)\in \mathbb{Z}^l$ is such that 
 $s_{i}-s_{i-1}\geq n-1+e$. The following proposition uses the notion of dominance order on the set of $l$-partitions. We write $\ulambda\unlhd \umu$ if $\ulambda$ and $\umu$ are $l$-partitions of the same integer and for all $c\in \{1,\ldots,l\}$ and $k\in \mathbb{Z}_{>0}$, we have 
 $$\sum_{1\leq j<c} |\lambda^{j}| +\sum_{1\leq i\leq k}\lambda_i^c \leq \sum_{1\leq j<c} |\mu^{j}|+\sum_{1\leq i\leq k}\mu_i^c.$$
\end{abs}

We obtain a proof of a generalization of a conjecture by Dipper, James and Murphy.  The ``only if'' part of the theorem is also proved in \cite[Prop. 5.9]{GL}. Our result allows the proof of the other part. 

\begin{Cor}[Generalized Dipper-James-Murphy's conjecture]\label{djmc}
Let ${\bf s}=(s_1,\ldots,s_l)\in \mathbb{Z}^l$ be  such that 
 $s_{i}-s_{i-1}\geq n-1+e$ for all $i=2,\ldots,l$. Then  $\ulambda$ is an Uglov  $l$-partition (that is a Kleshchev $l$-partition)  if and only if there exist  a sequence of elements   
$\mathfrak{i}_1,\ldots,\mathfrak{i}_n$ in $\mathbb{Z}/e\mathbb{Z}$ and  
integers $c_{\ulambda,\umu}$ for $\umu\in \Pi^l$  such that:
$$f_{\mathfrak{i}_1}\ldots f_{\mathfrak{i}_n}.\uemptyset =\sum_{\umu\in \Pi^l (n)}  c_{\ulambda,\umu} \umu,$$
with $c_{\ulambda,\ulambda}\neq 0$ and such that if $\ulambda \lhd \umu$ then  $c_{\ulambda,\umu}= 0$.

\end{Cor}
\begin{proof}
The fact that $c_{\ulambda,\ulambda}\neq 0$ is trivial. We show the rest of the proposition 
 by induction. Let  
 $(\gamma_1,\ldots,\gamma_r)$ be the removable nodes of the  staggered sequence of nodes of $\ulambda$ . We denote by $\mathfrak{j}\in \mathbb{Z}/e\mathbb{Z}$ the common residue of these nodes. Let $\ulambda'$ be the staggered  $l$-partition of $n-r$ obtained by removing the nodes 
 $(\gamma_1,\ldots,\gamma_r)$ from $\ulambda$. We thus have $\mathfrak{i}_1=\ldots=\mathfrak{i}_r=\mathfrak{j}$.  By induction, we have:
 $$f_{\mathfrak{i}_{r+1}}\ldots f_{\mathfrak{i}_{n}}.\uemptyset =\sum_{\umu'\in \Pi^l (n-r)}  c_{\ulambda',\umu'} \umu',$$
with $c_{\ulambda',\ulambda'}\neq 0$ and such that if $\ulambda' \lhd \umu'$ then  $c_{\ulambda',\umu'}= 0$.
If $\umu$ appears in the expansion of  $f_{\mathfrak{j}}^r \ulambda'$ with non zero coefficient, then by the definition of the staggered sequence, it is clear that we have $\umu\unlhd \ulambda$. 
Otherwise,  $\umu$ appears in the expansion of  $f_{\mathfrak{j}}^r \umu'$ with non zero coefficient 
 with  $\umu'$ not greater than $\ulambda'$ with respect to the dominance order and with $c_{\lambda',\mu'}$.

Assume  that $\ulambda  \lhd \umu$ then  there exist $c\in \{1,\ldots,l\}$ and $j\in \mathbb{Z}_{>0}$ such that $\lambda^m=\mu^m$ if $m<c$, $\lambda^c_{k}=\mu^c_{k}$ for $k\leq j$ and $\mu_j^c>\lambda_j^c$.
 Now the nodes $(\gamma_1,\ldots,\gamma_r)$ are the greater removable nodes of $\ulambda$ with respect to $<$. 
 This implies that necessarily $\ulambda' \lhd \umu'$, a contradiction.

\end{proof}

\begin{Rem}
To translate the above result in terms of the conventions used in   \cite{BK,DJM,Hu} (see Remark \ref{conven}), note  that an $l$-partition $\ulambda$ is greater to another one  $\umu$ with respect to the dominance order if and only if
 $\umu'$ is greater to $\ulambda'$ with respect to the dominance order. 
\end{Rem}

\begin{Rem}
It is natural to ask if the above corollary can be proved for the whole class of Uglov $l$-partitions.  In general, one can see that the above proof cannot be generalized. The case $l=2$ can however be treated using different techniques. This will be exposed in a future work. 

\end{Rem}

\subsection{A generalized Lascoux-Leclerc-Thibon algorithm}\label{llt}

We finish by giving an application of our main theorem. This is inspired by \cite{LLT} and \cite[Ch. 6]{GJ}.  
Let $v$ be an indeterminate and consider the affine quantum algebra $\mathcal{U}_v (\widehat{\mathfrak{sl}}_e)$ of type $A$. One can quantize the action of $\mathfrak{g}$ on the Fock space to obtain an action of 
 $\mathcal{U}_v (\widehat{\mathfrak{sl}}_e)$ on a quantized version of the Fock space: the $\mathbb{Q}(v)$-vector space with basis the set of  all the $l$-partitions $\Pi^l$. We refer to \cite[Ch. 6]{GJ} for details. 
 For each $\mathfrak{i}\in \mathbb{Z}/e\mathbb{Z}$ and $k\in \mathbb{Z}_{>0}$, let us denote by $f^{(k)}_i$ the associated divided power 
in $\mathcal{U}_v (\widehat{\mathfrak{sl}}_e)$ (see \cite[\S 6.1.5]{GJ}).

 Let $\preceq$ be the lexicographic order associated to $l$-partitions. This is defined as follows, let $(\ulambda , \umu)\in \Pi^l (n)^2$ with $n\in \mathbb{Z}_{>0}$ then we have $\ulambda \prec \umu$ if there exists $j\in \{1,\ldots,l)$ and $k\in \mathbb{Z}_{>0}$ such that $\lambda^t=\mu^t$ for all $t\in \{1,\ldots,j-1\}$, 
  $\lambda^j_m=\mu^j_m$ for all $m\in \{1,\ldots,k-1\}$ and $\lambda^j_k<\mu^j_k$. It is clear that  
 $\ulambda \lhd \umu$ implies 
 $\ulambda \prec \umu$.

 Now let 
  ${\bf s}=(s_1,\ldots,s_l)\in \mathbb{Z}^l$ be  such that 
 $s_{i}-s_{i-1}\geq n-1+e$ for all $i=2,\ldots,l$ and let  $\ulambda$ be an Uglov  $l$-partition (and thus an Uglov $l$-partition.)  
  We write :
  $$\underbrace{\mathfrak{i}_1,\ldots,\mathfrak{i}_{1}}_{a_1\text{ times}}
 \underbrace{\mathfrak{i}_2,\ldots,\mathfrak{i}_{2}}_{a_2\text{ times}}
 \ldots, \underbrace{\mathfrak{i}_m,\ldots,\mathfrak{i}_{m}}_{a_m\text{ times}}$$
for the associated staggered sequence of residues where $\mathfrak{i}_{j}\in \mathbb{Z}/e\mathbb{Z}$ and $a_j\in \mathbb{Z}_{>0}$ for all 
 $j\in \{1,\ldots,m\}$ and where we assume that $\mathfrak{i}_s \neq \mathfrak{i}_{s+1}$ for all $s\in \{1,\ldots,m-1\}$. By  Corollary \ref{djmc} together with the same proof as \cite[Thm 6.4.2]{GJ}, we obtain that 
 $$f_{\mathfrak{i}_1}^{(a_1)}\ldots f_{\mathfrak{i}_m}^{(a_m)}.\uemptyset =\ulambda + 
 \sum_{\umu\prec \ulambda}  c_{\ulambda,\umu} (v) \umu,$$
for Laurent polynomials  $c_{\ulambda,\umu} (v)$. As a consequence, mimicking the algorithm \cite[\S 6.4.9]{GJ}
 leads to an algorithm for the computation of the associated Kashiwara-Lusztig canonical basis which is a direct generalization of the LLT algorithm \cite{LLT}.

\vspace{0.5cm}
{\bf Address:}\\
\textsc{Nicolas Jacon}, Universit\'e de Reims Champagne-Ardenne, UFR Sciences exactes et naturelles, Laboratoire de Math\'ematiques EA 4535
Moulin de la Housse BP 1039, 51100 Reims, FRANCE\\  \emph{nicolas.jacon@univ-reims.fr}\\

\end{document}